\documentclass[preprint,12pt]{elsarticle}

\usepackage{amsmath}
\usepackage{array,overpic}
\usepackage{epsf,latexsym,amsfonts,amsbsy}
\usepackage{amsmath, amssymb, amsthm,amscd,amsxtra}
\usepackage{graphicx,epstopdf,epsfig}
\usepackage{color}

\def\be {\begin{equation}}
\def\ee {\end{equation}}
\def\ba {\begin{eqnarray}}
\def\ea {\end{eqnarray}}
\newcommand{\dss}{\displaystyle}
\newcommand{\tss}{\textstyle}
\newtheorem{theorem}{Theorem}

\newdefinition{rmk}{Remark}

\usepackage{amssymb}

\journal{J. Differential Equation}

\begin{document}

\begin{frontmatter}



\title{Bifurcation of small limit cycles\\
in cubic integrable systems  using higher-order analysis}  


\author{Yun Tian$^{a}$, Pei Yu$^{b,} 
$\footnote{Corresponding author.\\
\hspace*{0.25in}E-mail addresses: ytian22@shnu.edu.cn (Y. Tian), 
pyu@uwo.ca (P. Yu)}}


\address{$^a$Department of Mathematics, Shanghai Normal University, 
Shanghai, 200234, P.~R. China \\
$^b$Department of Applied Mathematics, Western University, 
London, Ontario, Canada N6A 5B7}

\begin{abstract}
In this paper, we present a method of higher-order analysis on bifurcation 
of small limit cycles around an elementary center of integrable systems 
under perturbations. This method is equivalent to higher-order Melinikov 
function approach used for studying bifurcation of limit cycles 
around a center but simpler.
Attention is focused on planar cubic polynomial systems 
and particularly it is shown that the system studied
by H. \.Zo{\l}\c{a}dek in the article 
{\it Eleven small limit cycles in a cubic vector field} 
(Nonlinearity 8, 843--860, 1995) can indeed 
have eleven limit cycles under 
perturbations at least up to $7$th order. Moreover, 
the pattern of numbers of limit cycles produced near 
the center is discussed up to $39$th-order perturbations,
and no more than eleven limit cycles are found.

\end{abstract}

\begin{keyword}
Bifurcation of limit cycles; Higher-order analysis; Darboux integral; 
Focus value.



\vspace{0.10in} 
\MSC 34C07 \sep 34C23

\end{keyword}

\end{frontmatter}



\section{Introduction}

Bifurcation theory of limit cycles is important for both theoretical 
development of qualitative analysis
and applications in solving real problems. It is closely related 
to the well-known Hilbert's 16th problem~\cite{Hilbert1900},
whose second part asks for the upper bound, 
called Hilbert number $ H(n) $, on the number of limit cycles 
that the following system,
\begin{eqnarray}\label{Eqn1}
\frac{dx}{dt} = P_n(x,y), \qquad \frac{dy}{dt} = Q_n(x,y),
\end{eqnarray}
can have, where $ P_n(x,y) $ and $ Q_n(x,y) $
represent $ n^{\rm th}$-degree polynomials in $ x$ and $ y$.
This problem has motivated many mathematicians 
and researchers in other disciplines to 
develop mathematical theories and methodologies in the areas of 
differential equations and dynamical systems. 
However, this problem has not been completely solved even for 
quadratic systems since Hilbert proposed the problem in 
the Second Congress of World Mathematicians in 1900. 
The maximal number of limit cycles obtained for some quadratic systems is 
$4$~\cite{Shi1980,ChenWang1979}. 
However, whether $ H(2)=4 $ is still open.
For cubic polynomial systems, many
results have been obtained on the lower bound of the number of limit cycles.
So far, the best result for cubic systems is
$ H(3) \ge 13$~\cite{LiLiu2010,LLY09}. Note that the $13$ limit cycles
obtained in~\cite{LiLiu2010,LLY09}
are distributed around several singular points.
 
When the problem is restricted to consider the maximum number of 
small-amplitude limit cycles, denoted by $M(n)$, bifurcating from a 
focus or a center in system (\ref{Eqn1}),
one of the best-known results is $M(2)=3$, which was obtained by
Bautin in 1952~\cite{Bautin1952}.
For $n=3$, a number of results in this research direction have been obtained. 
So far the best result for the number of small limit cycles around a 
focus is $9$~\cite{YC2009,Lloyd2012,CCMYZ2013}, 
and that around a center is 
$12$~\cite{YuTian2014}. 

One of powerful tools used for analyzing local bifurcation of limit cycles 
around a focus or a center
is normal form theory (e.g., see~\cite{GuckenheimerHolmes1993,ChowLiWang1994,
Kuznetsov1998,GazorYu2012}). 
Suppose system \eqref{Eqn1} has an elementary focus
or an elementary center at the origin.
With the computation methods using computer algebra systems
(e.g., see~\cite{HanYu2012, Yu1998,YuLeung2003,TianYu2013,TianYu2014}),
we obtain the normal form expressed in polar coordinates as
\begin{equation}\label{Eqn2}
\begin{array}{ll}  
\displaystyle\frac{dr}{dt} = r \, \big( v_0 + v_1 \, r^2 + v_2 \, r^4 
+ \cdots + v_k \, r^{2k} + \cdots \big), 
\\[1.0ex]
\displaystyle\frac{d \theta}{dt} 
= \omega_c + \tau_0 + \tau_1 \, r^2 + \tau_2 \, r^4 + \cdots 
+ \tau_k \, r^{2k} + \cdots ,  
\end{array} 
\end{equation}
where $r$ and $\theta $ represent the amplitude and phase of motion,
respectively.  $v_k \, (k=0,1,2, \cdots) $ is called the $k$th-order
focus value. $v_0$ and $\tau_0$ are obtained from linear analysis.
The first equation of (\ref{Eqn2}) can be used for studying
bifurcation and stability of limit cycles, while the
second equation can be used to determine the frequency of the bifurcating
periodic motion. Moreover, the coefficients $\tau_j$ can be used to
determine the order or critical periods of a center
(when $ v_j=0,\, j \ge 0$).

A particular attention has been paid to 
near-integrable polynomial systems, described in the form of 
\begin{equation} 
\begin{array}{ll}\label{Eqn3}
\displaystyle\frac{dx}{dt} = M^{-1}(x,y,\mu)H_y(x,y,\mu) 
+ \varepsilon \, p(x,y,\varepsilon,\delta), \\[1.0ex] 
\displaystyle\frac{dy}{dt} = -M^{-1}(x,y,\mu) H_x(x,y,\mu) 
+ \varepsilon \, q(x,y,\varepsilon,\delta),
\end{array}
\end{equation}
where $0 < \varepsilon \ll 1 $,
$\mu$ and $\delta$ are vector parameters;
$ H(x,y,\mu)$ is an analytic function in $x$, $y$ and $\mu$;
$p (x,y,\varepsilon,\delta) $ and $ q(x,y,\varepsilon,\delta) $ are
polynomials in $ x $ and $ y$,
and analytic in $\delta$ and $\varepsilon$.
$M(x,y,\mu)$ is an integrating factor of the unperturbed 
system \eqref{Eqn3}$|_{\varepsilon=0}$.

Suppose the unperturbed system \eqref{Eqn3}$|_{\varepsilon=0}$ 
has an elementary center.
Then, considering limit cycles bifurcation in system \eqref{Eqn3} 
around the center, we may use the normal form theory to 
obtain the first equation of \eqref{Eqn2} as follows: 
\begin{equation}\label{Eqn4} 
\begin{array}{rl} 
\displaystyle\frac{dr}{dt} 
= r \left[v_0(\varepsilon) + v_1(\varepsilon)r^2 
+ v_2(\varepsilon)r^4 + \cdots +v_i(\varepsilon)r^{2i}+\cdots \right],  
\end{array} 
\end{equation} 
where 
\begin{equation*}
v_i(\varepsilon) = \sum_{k=1}^\infty \varepsilon^k V_{ik}, 
\quad i=0,1,2, \dots,  
\end{equation*} 
in which $V_{ik}$ denotes the $i$th $\varepsilon^k$-order focus value, 
and will be used throughout this paper. 
Note that $v_i(\varepsilon)=O(\varepsilon)$
since the unperturbed system \eqref{Eqn3}$|_{\varepsilon=0}$ 
is an integrable system.
Further, because system \eqref{Eqn3} is analytic in $\varepsilon$,
we can rearrange the terms in \eqref{Eqn4}, and obtain 
\begin{equation}\label{Eqn5}
\frac{dr}{dt} = V_1(r)\, \varepsilon + V_2(r)\, \varepsilon^2 + \cdots 
+ V_k(r)\, \varepsilon^{k} + \cdots ,  
\end{equation} 
where 
\begin{equation}  \label{Eqn6} 
V_k(r) = \sum_{i=0}^\infty V_{ik}\, r^{2i+1}, \quad k=1,2, \dots. 
\end{equation}

Similarly, for the normal form of system \eqref{Eqn3} 
we have the $\theta$ differential equation, given by 
\begin{equation*}
\frac{d\theta}{dt} = T_0(r) + O(\varepsilon),
\end{equation*} 
with $T_0(0)\neq0$, and thus 
\begin{equation}\label{Eqn7}
\frac{dr}{d\theta}=\frac{V_1(r)\, \varepsilon + V_2(r)\, \varepsilon^2 
+ \cdots + V_k(r)\, \varepsilon^{k} + \cdots} {T_0(r) + O(\varepsilon)}.
\end{equation}
Assume the solution $r(\theta,\rho,\varepsilon)$ of \eqref{Eqn7}, satisfying
the initial condition $r(0,\rho,\varepsilon)=\rho$, is given 
in the form of 
\begin{equation*}
r(\theta,\rho,\varepsilon)=r_0(\theta,\rho)+r_1(\theta,\rho)\varepsilon
+r_2(\theta,\rho)\varepsilon^2 + \cdots 
+ r_k(\theta,\rho)\varepsilon^k + \cdots,
\end{equation*}
with $0<\rho\ll1$. Then, $r_0(0,\rho)=\rho$ and $r_k(0,\rho)=0$, 
for $k\ge1$.

If there exists a positive integer $K$ such that $V_k(r)\equiv0$, $1\le k< K$,
and $V_K(r)\not\equiv0$, then it follows from \eqref{Eqn7} that 
$$r_0(\theta,\rho)=\rho,\quad r_k(\theta,\rho)=0,
\quad 1\le k< K,\quad \mbox{and}
\quad r_{K}(\theta,\rho)=\frac{V_K(\rho)}{T_0(\rho)}\theta.$$
Thus, the displacement function $d(\rho)$ of system \eqref{Eqn7} 
can be written as
\begin{equation}\label{Eqn8}
d(\rho)=r(2\pi,\rho,\varepsilon)-\rho=2\pi\frac{V_K(\rho)}{T_0(\rho)}\varepsilon^K+O(\varepsilon^{K+1}).
\end{equation}
Therefore, if we want to determine the number of small-amplitude limit cycles 
bifurcating from the center in system \eqref{Eqn3}, we only need to study 
the number of isolated zeros of $V_K(\rho)$ for $0<\rho\ll 1$, 
and have to obtain the expression of the first non-zero coefficient
$V_K(r)$ in \eqref{Eqn5} by computing $V_{iK}$, for $i\ge0$.  

The above discussions show that the basic idea of using focus values 
is actually the same as that of the Melnikov function method.
Using $H(x,y)=h$ to parameterize the section (i.e. the 
Poincar\'{e} map),
we obtain the displacement function of \eqref{Eqn3}, given by
\begin{equation}\label{Eqn9}
d(h)= M_1(h)\varepsilon + M_2(h)\varepsilon^2 + \cdots +
M_k(h)\varepsilon^k + \cdots,
\end{equation}
where
\begin{equation}\label{Eqn10}
M_1(h) \!= \!\! \dss\oint_{H(x,y,\mu)=h} \hspace*{-0.30in} M(x,y,\mu) 
\big[ q(x,y,0,\delta) \, dx - p(x,y,0,\delta) \, dy \big], 
\end{equation}
evaluated along closed orbits $H(x,y,\mu)=h$ for $h\in(h_1,h_2)$.
Then, we can study the first non-zero Melnikov function
$M_k(h)$ in \eqref{Eqn9} to determine the number of 
limit cycles in system \eqref{Eqn3}. 
In the following, we remark on the comparison of the Melnikov function 
method and the method of normal forms (or focus values). 

\begin{rmk}\label{Rem1}

\begin{enumerate}
\item[{\rm (1)}]
Let $H=h$,  $0<h-h_1\ll1$ define closed orbits around 
the center of system \eqref{Eqn5}$|_{\varepsilon=0}$.
It is easy to see that for any integer $K\ge1$,
equation \eqref{Eqn8} holds
if and only if
$M_k(h)\equiv0$, $1\le k<K$ and $M_K(h)\not\equiv0$ in \eqref{Eqn9}.
Moreover, $V_K(\rho)$ for $0<\rho\ll1$ and $M_K(h)$ for $0<h-h_1 \ll 1$ 
have the same maximum number of isolated zeros. 

\item[{\rm (2)}]
As we can see, $V_k(r)$ can be obtained by the computation of normal forms 
or focus values.

\item[{\rm (3)}]
In particular, when the original system is not a Hamiltonian system but
an integrable system, then even computing the coefficients of the
first-order Melnikov function is much more involved than the computation
of using the method of normal forms. 

\item[{\rm (4)}]
However, the method of normal forms (or focus values) is restricted to Hopf and 
generalized Hopf bifurcations,
while the Melnikov function method can also be applied to study bifurcation
of limit cycles from homoclinic/heteroclinic loops or any closed orbits.
\end{enumerate}
\end{rmk}

When we apply the method of normal form computation,
some unnecessary perturbation parameters are involved 
in the computation of high-order focus values,
which could be extremely computation demanding (in both time and memory), 
and makes it much more difficult to solve the problem. 
Meanwhile, before we use the first non-zero coefficient 
$V_K(r)$ in \eqref{Eqn5} to find limit cycles,
we need to prove $V_k(r)\equiv0$, $1\le k < K$.
The unnecessary parameters involved could greatly 
increase the difficulty of proving that.

In this paper, without loss of limit cycles,
we introduce a linear transformation to eliminate unnecessary 
parameters from system \eqref{Eqn3}. 
With less parameters in \eqref{Eqn3}, we can use the approximation of 
first integrals to prove $V_k(r)\equiv0$.
The idea will be illuminated in Section 2.

We will apply our method to study the bifurcation of small-amplitude limit 
cycles in the system 
\begin{equation}\label{Eqn11}
\begin{array}{ll}  
\displaystyle\frac{dx}{dt} = a + \dss\frac{5}{2}\, x + xy + x^3 
+ \sum_{k=1}^n \varepsilon^k p_k(x,y), \\[-1.0ex] 
\displaystyle\frac{dy}{dt} = -2 a x + 2 y - 3 x^2 + 4 y^2 - a x^3 + 6 x^2 y 
+  \dss\sum_{k=1}^n \varepsilon^k q_k(x,y),  
\end{array} 
\end{equation} 
where 
\begin{equation}\label{Eqn12} 
p_k(x,y) = a_{00k} + \sum_{i+j=1}^3 a_{ijk}\, x^i y^j, \qquad 
q_k(x,y) = b_{00k} + \sum_{i+j=1}^3 b_{ijk}\, x^i y^j, 
\end{equation} 
in which $a_{ijk}$ and $b_{ijk}$ are $\varepsilon^k$th-order coefficients 
(parameters). 
The unperturbed system \eqref{Eqn11}$|_{\varepsilon=0}$ has a rational 
Darboux integral \cite{Zoladek1995}, 
\begin{equation}\label{Eqn13}
H_0=\frac{f_1^5}{f_2^4}=\frac{(x^4+4x^2+4y)^5}{(x^5+5x^3+5xy+5x/2+a)^4},
\end{equation}
with the integrating factor $M=20f_1^4f_2^{-5}$.
It can be shown that for $a<-2^{5/4}$, system \eqref{Eqn11}$|_{\varepsilon=0}$
has a center at ${\rm E_0} =(- \frac{a}{2},- \frac{a^2+2}{4})$. 
The system \eqref{Eqn11}$|_{\varepsilon=0}$ was proposed 
in \cite{Zoladek1995}, and it was claimed that this system 
could have $11$ limit 
cycles around the center by studying the second-order Melnikov 
function. Later, Yu and Han applied the normal form computation method and 
got only $9$ limit cycles around ${\rm E_0}$~\cite{YuHan2011}
by analyzing the $\varepsilon$- and $\varepsilon^2$-order focus values. 
Recently, it has been shown \cite{TianYu2016} 
that errors are made in \cite{Zoladek1995} for choosing 12 integrals as 
the basis of the linear space of corresponding Melnikov functions of system 
\eqref{Eqn11}$|_{\varepsilon=0}$.
In fact, among the 12 chosen integrals, 
two of them can be expressed as linear combinations of the other ten integrals,
and therefore only $9$ limit cycles can exist, agreeing with that shown in~\cite{YuHan2011}.   

The rest of the paper is organized as follows. 
In the next section, we consider system \eqref{Eqn3}, and 
construct a transformation to reduce the number of perturbation 
parameters, which greatly simplifies the analysis in the following section. 
Section 3 is devoted to the computation of higher 
$\varepsilon^k$-order focus values and the existence of 
$11$ limit cycles in system \eqref{Eqn11}, which needs 
computing at least $\varepsilon^7$-order focus values.
Finally, conclusion is drawn in Section 4.

\section{Preliminaries}

The method of focus values (or normal forms) is one of important and powerful 
tools for the study of small-amplitude limit cycles generated from Hopf 
bifurcation. In general, a sufficient number of focus values would be needed 
if one wants to find more small-amplitude limit cycles.
One main challenge is that the computation of focus values becomes
more and more difficult as the order of focus values goes up.
That is why computer algebra systems such as Maple and Mathematica have 
been used for computing the focus values 
to improve the computational efficiency (e.g. see~\cite{TianYu2013,TianYu2014}).
Another approach is 
to eliminate certain parameters from the system, 
which is the method we shall develop here for near-integrable systems.

In most studies of near-integrable systems,
full perturbations like those polynomials $p(x,y,\varepsilon,\delta)$ 
and $q(x,y,\varepsilon,\delta)$ given in system \eqref{Eqn3} are considered. 
The parameter vector $\delta$ usually represents the coefficients
in $p$ and $q$. When normal forms are used to study small limit cycles,
it is easy to get and solve the focus values of $\varepsilon$ order
(coefficients in $V_1(r)$), because they are linear functions of the 
system parameters, namely the coefficients in $p(x,y,0,\delta)$ 
and $q(x,y,0,\delta)$.
For the $\varepsilon^k$-order focus values (coefficients in $V_k(r)$), 
more parameters would be involved in the computation.
One can observe that some parameters are not necessary 
for obtaining the maximum number of limit cycles, 
and they only increase the difficulty in finding limit cycles.

When  the first $n$ functions $V_k(r)$ in \eqref{Eqn5}, 
$1\le k\le n$
are applied to studying 
bifurcation of limit cycles, in order to 
remove unnecessary parameters 
without reducing the number of limit cycles,  
we may use the following transformation: 
\begin{equation}\label{Eqn14}
\left\{\begin{split}
&x\rightarrow x+e_{1}(\varepsilon)x+e_{2}(\varepsilon)y+e_{3}(\varepsilon),\\
&y\rightarrow y+e_{4}(\varepsilon)x+e_{5}(\varepsilon)y+e_{6}(\varepsilon),\\
&t\rightarrow t+e_{7}(\varepsilon)t,\\
&\mu\rightarrow \mu+e_8(\varepsilon),
\end{split}
\right.
\end{equation}
where   
$$
e_{i}(\varepsilon)=e_{i1}\varepsilon+e_{i2}\varepsilon^2+\cdots 
+e_{in}\varepsilon^n,\quad i=1,\cdots,8.
$$
Note that \eqref{Eqn14}$|_{\varepsilon=0}$ is an identity map.
Thus, \eqref{Eqn14} keeps the unperturbed system of \eqref{Eqn3} unchanged.
Furthermore, the new system obtained by using \eqref{Eqn14} can
be still written in the same form of \eqref{Eqn3}.
So we only need to find proper $e_{i}(\varepsilon)$'s to get 
simpler perturbation functions without loss of generality.

To illustrate how to obtain $e_i(\varepsilon)$, 
we take system \eqref{Eqn11} as an example.
The coefficients $a_{ijk}$ and $b_{ijk}$ in \eqref{Eqn11} are the parameters.
Substituting the transformation \eqref{Eqn14} 
into system \eqref{Eqn11} yields
\begin{equation}\label{Eqn15}
\begin{array}{ll}  
\displaystyle\frac{dx}{dt} = a + \dss\frac{5}{2}\, x + xy + x^3 
+ \sum_{k=1}^n \varepsilon^k \tilde p_k(x,y)+o(\varepsilon^{n}), \\[-1.0ex] 
\displaystyle\frac{dy}{dt} = -2 a x + 2 y - 3 x^2 + 4 y^2 - a x^3 + 6 x^2 y 
+  \dss\sum_{k=1}^n \varepsilon^k \tilde q_k(x,y) +o(\varepsilon^{n}),  
\end{array} 
\end{equation} 
where
\begin{equation}\label{Eqn16} 
\tilde p_k(x,y) 
= \tilde a_{00k} + \sum_{i+j=1}^3 \tilde a_{ijk}\, x^i y^j, \qquad 
\tilde q_k(x,y) 
= \tilde b_{00k} + \sum_{i+j=1}^3 \tilde b_{ijk}\, x^i y^j.
\end{equation}

Obviously, the coefficients $\tilde a_{ijk}$ and $\tilde b_{ijk}$ 
in \eqref{Eqn16} are linear in $e_{mk}$, $m=1,\ldots,8$.
Let $E_k=(e_{1k},e_{2k},\cdots,e_{8k})^T$.
For any $1\le k\le n$, $\tilde a_{ijk}$ and 
$\tilde b_{ijk}$ can be written in the form of 
$$
\tilde a_{ijk} = A_{ij}E_k +\eta_{ijk},\quad
\tilde b_{ijk} = B_{ij}E_k +\zeta_{ijk},
$$
where $A_{ij}$ and $B_{ij}$ are $1\times8$ matrices, 
and $\eta_{ijk}$ and $\zeta_{ijk}$, given by 
\begin{equation}\label{Eqn17}
\begin{split}
&\eta_{ijk} = \eta_{ijk}(E_1,\cdots,E_{k-1},
a_{ml1},\cdots,a_{mlk}, b_{ml1},\cdots,b_{mlk}),\\
&\zeta_{ijk} = \zeta_{ijk}(E_1,\cdots,E_{k-1},
a_{ml1},\cdots,a_{mlk}, b_{ml1},\cdots,b_{mlk}) ,
\end{split}
\end{equation}
are polynomials in $e_{ml}$, $1\le l  \le k-1$, 
and the coefficients in the perturbation functions \eqref{Eqn12}. 

Note that $A_{ij}$ and $B_{ij}$ are not dependent on $k$.
We hope that we can find some proper values for $e_{ik}$ 
to make some of the coefficients 
$\tilde a_{ijk}$ and $\tilde b_{ijk}$ vanish or satisfy some conditions,
so that the computation of the focus values would become easier.
For instance, we can choose for $1\le k\le n$, 
\begin{equation}\label{Eqn18}
\begin{array}{ll}
&\tilde a_{10k}= \tilde a_{01k}= \tilde a_{20k}= 
\tilde a_{11k}= \tilde a_{02k}= \tilde a_{30k}= 0,\\[1ex]
\hspace*{-0.74in} \mbox{and}\qquad \qquad & 
\tilde a_{p_k}\triangleq \tilde p_k(-\frac{a}{2},-\frac{a^2+4}{4})=0,\quad
\tilde a_{q_k}\triangleq \tilde q_k(-\frac{a}{2},-\frac{a^2+4}{4})=0.
\end{array}
\end{equation}
The last two equations in \eqref{Eqn18} keep the equilibrium 
of system \eqref{Eqn11} in a neighborhood of 
$\rm E_0$ with radius $o(\varepsilon^n)$.
A direct computation yields 
\begin{equation}\label{Eqn19} 
\begin{array}{rl}
\tilde a_{10k} =&\!\!\! 2ae_{2k}+e_{6k}+\frac{5}{2}e_{7k}+\eta_{10k},\quad
\tilde a_{01k} = \frac{1}{2}e_{2k}+e_{3k}+\eta_{01k},\\[0.5ex]
\tilde a_{20k} =&\!\!\! 3e_{2k}+3e_{3k}+e_{4k}+\eta_{20k},\quad
\tilde a_{11k} = e_{5k}+e_{7k}+\eta_{11k},\\[0.5ex]
\tilde a_{02k} =&\!\!\! -3e_{2k}+\eta_{02k},\quad
\tilde a_{30k} = 2e_{1k}+ae_{1k}+e_{7k}+\eta_{30k},\\[1ex]
\tilde a_{p_k} =&\!\!\!  -\frac{1}{4}a(4+a^2)e_{1k} 
- \frac{1}{8}(4+a^2)(2+a^2)e_{2k} + \frac{1}{4}(4+a^2)e_{3k}\\[1ex]
         &\!\!\!+\frac{1}{4}a^2e_{4k} + \frac{1}{8}a(2+a^2)e_{5k} 
- \frac{1}{2}ae_{6k} + e_{8k} +\tilde\eta_{k} ,\\[1ex]
\tilde a_{q_k} =&\!\!\! -\frac{1}{8}a^2(16+3 a^2)e_{1k} 
- \frac{1}{16}a(16+3 a^2)(2+a^2)e_{2k} \\[1.0ex] 
         &\!\!\!  + \frac{1}{4}a(16+3a^2)e_{3k}
+\frac{1}{4}a(4+a^2)e_{4k} + \frac{1}{8}(4+a^2)(2+a^2)e_{5k} \\[1ex]
&\!\!\!  -\frac{1}{4}(4+a^2)e_{6k}
+ \frac{1}{8}a(a^2+8)e_{8k}+\tilde\zeta_{k},\\
\end{array}
\end{equation}
where $\tilde\eta_{k}$ and $\tilde\zeta_{k}$ are also functions 
in $\eta_{ijl}$ and $\zeta_{ijl}$ with $1\le l  \le k-1$, 
respectively.

Because
$$
\det\left[\frac{\partial(\tilde a_{10k},\tilde a_{01k},\tilde a_{20k},
\tilde a_{11k}, \tilde a_{02k},\tilde a_{30k},\tilde a_{p_k},\tilde a_{q_k})}
{\partial(e_{1k},e_{2k},e_{3k},e_{4k},e_{5k},e_{6k},e_{7k},e_{8k})}\right]
=\frac{3}{4}(32-a^4)<0$$
for $a<-2^{-5/4}$,
we can solve \eqref{Eqn19} for $e_{mk}$ to obtain 
$$
e_{mk}=e_{mk}(\eta_{10k},\eta_{01k},\eta_{20k},\eta_{11k},\eta_{02k},
\eta_{30k},\tilde\eta_k,\tilde\zeta_{k}),\quad 
1\le m\le 8,
$$
which can be rewritten by using \eqref{Eqn17} as
$$
e_{mk}=\tilde e_{mk}(E_1,\cdots,E_{k-1},
a_{ij1},\cdots,a_{ijk}, b_{ij1},\cdots,b_{ijk}).
$$
Note that $e_{m1}$ only depends on $a_{ij1}$ and $b_{ij1}$.
Therefore, for all $1\le m\le 8$, $1\le k\le n$, $e_{mk}$ can 
be expressed as a polynomial
in $a_{ijl}$ and $b_{ijl}$, $1\le l\le k$.
In other words, \eqref{Eqn18} has solutions for all $1\le k\le n$.

Thus, without loss of generality, we assume that \eqref{Eqn12}
takes the following form, 
\begin{equation}\label{Eqn20}
\begin{array}{ll}  
p_k(x,y) = \!\!\!\! & a_{00k} + 
a_{21k} x^2 y + a_{12k} x y^2 + a_{03k} y^3, \\[0.5ex]   
q_k(x,y) = \!\!\!\! & b_{00k} 
+ b_{10k} x   + b_{01k} y
          +b_{20k} x^2 + b_{11k} x y + b_{02k} y^2 \\[0.5ex] 
\!\!\!\!& +b_{30k} x^3 + b_{21k} x^2 y + b_{12k} x y^2 + b_{03k} y^3,  
\end{array} 
\end{equation} 
with
\begin{equation}\label{Eqn21}
\begin{array}{ll}  
a_{00k}= \!\!\!\! & \textstyle\frac{1}{64} (a^2+2)
\big[ (a^2+2)^2\, a_{03k} + 2 a (a^2+2)\, a_{12k} +4 a^2 a_{21k} \big],  
\\[1.0ex] 
b_{00k} = \!\!\!\! & \textstyle\frac{1}{64}
\big\{ 8 a^3 b_{30k}
       \!+\! 16 a (2 b_{10k} \!-\! a b_{20k})
       \!+\! 4 (a^2 \!+\! 2) (4 b_{01k} \!-\! 2 a b_{11k} 
\!+\! a^2 b_{21k}) \\[1.0ex] 
\!\!\!\! & \quad \ 
-\,(a^2+2)^2 \big[ 4 b_{02k} -2 a b_{12k} -(a^2+2) b_{03k} \big]
\big\}. 
\end{array}
\end{equation}

As mentioned in Section 1,
to find limit cycles around $\rm{E_0}$ in system \eqref{Eqn11},
we apply normal form theory to compute the focus values and then solve 
the multivariate polynomial equations based on the focus values. 
Particularly, we have
\begin{equation}\label{Eqn22}
\begin{array}{ll}  
b_{01k} = \!\!\!\! & \textstyle\frac{1}{16} 
\big[ 4 a  (2 b_{11k}-a b_{21k}) - (a^2+2)^2 (a_{12k} +3 b_{03k}) 
\\[1.0ex] 
\!\!\!\! & + \,4 (a^2+2) (2 b_{02k} -a a_{21k} -a b_{12k}) \big],  
\end{array}
\end{equation} 
solved from the zeroth-order focus value $V_{0k}=0$, where  
\begin{equation}\label{Eqn23}
\begin{array}{rl}  
V_{0k} & \!\!\!\! = \tss\frac{1}{32} \big\{ 
16\, b_{01k} - 4 a  (2 b_{11k}-a b_{21k}) + (a^2+2)^2 (a_{12k} +3 b_{03k})  
\\ [1.0ex] 
& \!\!\!\!  \qquad 
-\, 4 (a^2+2) (2 b_{02k} -a a_{21k} -a b_{12k}) \big]. 
\end{array}
\end{equation} 
Higher-order focus values are relatively complex, and we shall
study them in Section 3. 

When we want to use focus values $V_{iK}$ in $V_K(r)$, $i=0,1,2,\cdots$,
to study limit cycles, we first need to show $V_k(r)\equiv0$, 
$1\le k<K$, or 
$\frac{dr}{dt}  = O(\varepsilon^{K})$ in \eqref{Eqn5}. In order to prove this, 
we use the approximation of first integrals, and claim that 
if there exists an analytic function $H_K(x,y,\varepsilon)$ such that
\begin{equation}\label{Eqn24} 
(M^{-1} H_y+\varepsilon p) \frac{\partial H_K}{\partial x}
+ (- M^{-1} H_x+\varepsilon q)\frac{\partial H_K}{\partial y}
= O(\varepsilon^K),
\end{equation} 
then $\frac{dr}{dt}=O(\varepsilon^K)$.
This claim can be easily proved 
by using the closed contour $H_K=h$ as the parameter
to express the displacement function.

Usually, like that considered in \cite{YuTian2014,YuHan2011} 
the method of focus values is used only to prove
how many limit cycles around the equilibrium point
that system \eqref{Eqn3} can have. 
Combining it with the approximation of first integrals,
we can obtain
the maximal number of small limit cycles for parameters 
in a neighborhood of critical conditions.
Furthermore, if the focus values are linear functions in  
parameters, we have a global result as follows. 

\begin{theorem}[Theorem 2.4.3 in \cite{Han2013}]\label{Thm1}
Consider system \eqref{Eqn5} and assume $V_{k}(r)\equiv0$, $1\le k<K$. 
Suppose that for an integer $m\ge 1$,
each $V_{iK}$, $0\le i<m$ is linear in $\delta$,
and further the following two conditions hold:

\vspace{0.05in} 
{\rm (i)}  $ \mbox{\rm rank} \Big[ \frac{\partial(V_{0K},\cdots,V_{m-1,K})}
{\partial(\delta_1,\cdots,\delta_m)} \Big] =m$,

\vspace{0.05in} 
{\rm (ii)} $V_{K}(r)\equiv0,\ {\rm if}\ V_{iK}=0, i=0,1,\cdots,m-1$. 

\vspace{0.05in}  
\noindent 
Then, for any given $N>0$, there exist $\varepsilon_0>0$ and a 
neighborhood $V$ of the origin such that system \eqref{Eqn3} 
has at most $m-1$ limit cycles in $V$ for 
$0<|\varepsilon|<\varepsilon_0$, $|\delta|\le N$. Moreover, $m-1$ limit 
cycles can appear in an arbitrary neighborhood of the origin for some 
values of $(\varepsilon,\delta)$. 
\end{theorem}

The above theorem can be proved following the proof given 
in \cite{Han2013} with a minor modification. So the proof is omitted here.

\section{Higher-order analysis leading to $11$ limit cycles 
in system \eqref{Eqn11}}

In this section, we focus on system \eqref{Eqn11}
and show that it can have $11$ limit cycles by using 
perturbations at least up to $7$th order. In the following,  
we will use the transformed system \eqref{Eqn11} with the 
simplified perturbations given in \eqref{Eqn19} for the analysis. 

In order to compute the focus values of this system, we first shift 
the equilibrium of system \eqref{Eqn11}, 
$(x,y)=(-\frac{a}{2}+o(\varepsilon^n),
-\frac{a^2+2}{4}+o(\varepsilon^n))$ 
to the origin and then use a computer algebra system
and software package (e.g., the Maple program in~\cite{Yu1998}) to 
obtain the focus values in terms of the parameters 
$a$, $a_{ijk}$ and $b_{ijk}$. 
We shall give detailed analysis for the first few lower-order 
focus values, 
and then summarize the results obtained from 
higher-order analysis. 

For convenience, define the vectors:  
\begin{equation}\label{Eqn25} 
\begin{array}{cl} 
W_k^8 & \!\!\! = \big(V_{1k}, V_{2k}, \cdots, V_{8k} \big), \\[0.5ex] 
W_k^9 & \!\!\! = \big(V_{1k}, V_{2k}, \cdots, V_{9k} \big), \\[0.5ex] 
W_k^{10} & \!\!\! = \big(V_{1k}, V_{2k}, \cdots, V_{10k} \big), \\[1.0ex] 
S_k^8 & \!\!\! =  
\big(b_{10k}, b_{20k}, b_{11k}, b_{02k}, b_{30k}, b_{21k}, b_{12k}, 
b_{03k} \big), \\[0.5ex]  
S_k^9 & \!\!\! =   
\big(b_{10k}, b_{20k}, b_{11k}, b_{02k}, b_{30k}, b_{21k}, b_{12k}, 
b_{03k}, a_{03k} \big), \\[0.5ex]  
S_k^{10} & \!\!\! =  
\big(b_{10k}, b_{20k}, b_{11k}, b_{02k}, b_{30k}, b_{21k}, b_{12k}, 
b_{03k}, a_{03k}, a_{12(3m)} \big), 
\end{array} 
\end{equation} 
where in $ S_k^{10}$, $k=7m$ for Case (A) and $k=13m$ for Case (B) 
($m \! \ge \! 1$, integer) to be considered in Sections 3.4 and 3.5;  
and the determinants:  
\begin{equation}\label{Eqn26} 
\begin{array}{rl}  
\det_k^{8} = \!\!\!& 
\det \! \left[ \textstyle\frac{\partial W_k^8} {\partial S_k^8 } \right]\!,
\quad  
\det_k^{9} = \det \! \left[ \textstyle\frac{\partial W_k^9} 
{\partial S_k^9} \right]\!, \quad  
\det_k^{10} = \det \! \left[ \textstyle\frac{\partial W_k^{10}} 
{\partial S_k^{10} } \right]\!;    
\end{array} 
\end{equation} 
and the functions:  
\begin{equation}\label{Eqn27}   
\begin{array}{ll}  
F_1 = \!\!\!\! & 
-\, \textstyle\frac{373423834799904305184768}{5\, a^{36} (a^4-32)^8} ,  
\\[1.0ex]
F_2 = \!\!\!\! & 
\textstyle\frac{3013505105717894236809449177088}{5\, a^{45} (a^4-32)^9}, 
\\[1.0ex] 
F_3 = \!\!\!\! & 
-\, \textstyle\frac{57397219210893210316046010501071634432}
{a^{55} (a^4-32)^{10}}, 
\\[1.0ex] 
F_4 = \!\!\!\! & 
-\, \textstyle\frac{279638476916415193342384256641414767487418}  
{a^{66} ( a^4 - 32)^{11}}, 
\\[2.0ex]  
G_1 = \!\!\!\! & -\, \frac{258237837}{32\, a^9 (a^4-32)}, \\[1.0ex] 
G_2 = \!\!\!\! & \frac{23476167}{64\, a^{11} (a^4-32)^2} 
      (57697 a^4-35728 a^2-88704), \\[1.0ex]  
G_3 = \!\!\!\! & -\, \frac{23476167}{1024\, a^{13} (a^4-32)^3}  
      (2304313595 a^8 \!-\! 1702233920 a^6 \!-\! 11829269248 a^4 \\[0.5ex] 
\!\!\!\! & \hspace{1.2in} -\, 39211065344 a^2+8642101248), \\[2.0ex] 
G_4 = \!\!\!\! & -\, \textstyle\frac{75246080}{ a^{10} (a^4-32)}, \\[1.0ex] 
G_5 = \!\!\!\! & \textstyle\frac{9405760}{3\, a^{12}(a^4-32)^2} 
      (75767 a^4-46944 a^2-96768), \\[1.0ex] 
G_6 = \!\!\!\! & -\,\textstyle\frac{180880}{3\, a^{14} (a^4-32)^3} 
      (11681524055 a^8-8555309984 a^6 -56944147200 a^4  \\ [0.5ex] 
\!\!\!\! & \hspace{1.0in} -\, 204210659328 a^2+30640177152), \\[2.0ex]  
G_7 = \!\!\!\! & \textstyle\frac{2006968901247765}{2883584\, 
a^{11} (a^4-32)}, 
\\[1.0ex] 
G_8 = \!\!\!\! & -\, \textstyle\frac{154382223172905}{2883584\, 
a^{13}(a^4-32)^2} (48667 a^4-30160 a^2-52416), 
\\[1.0ex] 
G_9 = \!\!\!\! & \textstyle\frac{66163809931245}{46137344\,a^{15} 
(a^4-32)^3} 
         (6314158847 a^8-4591849024 a^6 \\[0.5ex] 
\!\!\!\! &  \hspace{0.5in} -\,29599122432 a^4 - 112639700992 a^2+11915624448),  
\end{array} 
\end{equation}  
Note that $ F_i \ne 0,\, i=1,2,3$, and 
$G_i \ne 0, \, i=1,2, \dots, 9$, since 
$a^4 \!-\! 32 \!>\! 0$ for $ a \!<\! -2^{-5/4}$.

\subsection{$\varepsilon$- and $\varepsilon^2$-order analysis}

The $\varepsilon$-order focus values 
$V_{11}$, $V_{21}$, $\cdots$, $V_{111}$ are obtained 
by using the algorithm and Maple program developed in \cite{Yu1998}.
Their expressions are lengthy, and here we only present the first one 
for brevity, 
\begin{equation}\label{Eqn28} 
\begin{array}{cl} 
V_{11} & \!\!\!\!  = -\, \tss\frac{1}{64 a (a^4-32)} 
\big\{ 6912 b_{101} -5760 a \, b_{201} 
\\[1.0ex] 
& \!\!\!  \ \ 
+16 (a^4-36 a^2+40) \, b_{111} 
+48 a (a^4 +36 a^2 +160)\, b_{021} \\[1.0ex] 
& \!\!\!  \ \
+3456 a^2\, b_{301} -24 a (a^4-12 a^2+40)\, b_{211}  \\[1.0ex] 
& \!\!\!  \ \
-16 (3 a^6+68 a^4+300 a^2+20)\, b_{121}  \\[1.0ex]  
& \!\!\!  \ \ 
-24 a (3 a^6+65 a^4+300 a^2+224)\, b_{031} \\[1.0ex]  
& \!\!\!  \ \
+27 (a^2+2) (7 a^6+82 a^4+320 a^2+128)\, a_{031} \\[1.0ex] 
& \!\!\!  \ \
+8 a (24 a^6+223 a^4+1140 a^2+1180)\, a_{121} \\[0.5ex]  
& \!\!\!  \ \
+8 (21 a^6-73 a^4+480 a^2+320)\, a_{211}  \big\}.  
\end{array} 
\end{equation} 
It is noted that all $V_{i1}$'s are linear polynomials in
$a_{ij1}$ and $b_{ij1}$. It can be shown that 
\begin{equation}\label{Eqn29} 
\textstyle\det_1^{8} = F_1 \ne 0, \quad 
\det_1^{9} = F_2 \ne 0, \quad 
\det_1^{10} = 0. 
\end{equation}
In fact, with the solution of $S_1^8$ solved from $W_1^8=0$, we obtain 
\begin{equation}\label{Eqn30} 
\begin{array}{ll} 
V_{91} = G_1 \, a_{031}, \quad 
V_{101} = G_2 \, a_{031}, \quad 
V_{111} = G_3 \, a_{031}, 
\end{array} 
\end{equation} 
where $G_i$'s are given in \eqref{Eqn27}. Noticing 
$G_1 \ne 0$ for $a < - 2^{5/4}$, we have 
$V_{91} \ne 0$ if $a_{031} \ne 0$. 
Moreover, ${\det}_1^8 \ne 0$ and \eqref{Eqn23}
indicate that perturbing $W_1^8$ and $V_{01}$ 
around the solutions $S_1^8$ and $b_{011}$ (see \eqref{Eqn22}) does yield 
$9$ small limit cycles around the equilibrium ${\rm E_0}$. 

It is seen from \eqref{Eqn30} that 
$V_{91} \!=\! V_{101} \!=\! V_{111} \!=\! 0$ for $a_{031} \!=\! 0$. 
For convenience, define the critical condition ${\rm S_{1c}^8} $, 
satisfying \eqref{Eqn22}, $W_1^8 \!=\! 0$ and $a_{031} \!=\! 0$, as 
\begin{equation}\label{Eqn31}
{\rm S_{1c}^8\! :} \ \left\{\! \begin{array}{rl}
b_{011} = &\!\!\! C_1 \, a_{121} -\frac{9}{8}a^3a_{211}, \\[1.0ex]
a_{031} = & \!\!\! b_{211} = 0, \ \ b_{121} = \frac{7}{2}\, a_{211}, \ \
b_{021} = -\,6\, a_{121}, \ \ b_{031} = \frac{8}{3}\, a_{121}, \\[1.0ex] 
b_{111} = & \!\!\! 9 a\, a_{121} + \frac{9}{2}\, a_{211}, \hspace*{0.265in} 
b_{101} = C_2 \, a_{121} + C_3 \, a_{211}, \\[1.0ex] 
b_{201} = & \!\!\! C_4 \, a_{121} + C_5 \,a_{211},  \quad 
b_{301} = C_6 \, a_{121} + C_7 \, a_{211}, 
\end{array}  \hspace*{-0.20in} 
\right. 
\end{equation}
where $C_i$'s are given in Appendix A.  

We have the following result.

\begin{theorem}\label{Thm2}
The equilibrium ${\rm E_0}$ of system \eqref{Eqn11}
is a center up to $\varepsilon$-order, 
i.e. all $\varepsilon$-order focus values vanish  
if and only if the condition ${\rm S_{1c}^8}$ holds.
Furthermore, there exist at most $9$ small limit cycles
around ${\rm E_0}$ for all parameters $a_{ij1}$ and 
$b_{ij1}$, and $9$ small limit cycles can be obtained for some parameter
values near ${\rm S_{1c}^8}$. 
\end{theorem}

\begin{proof} 
The existence of $9$ small limit cycles has been shown under 
the solution $ S_1^8$ with 
$ a_{031} \! \ne \! 0$ and $ {\rm det}_1^8 \! \ne \! 0$. 
It is obvious that the critical condition ${\rm S_{1c}^8}$ 
is {\it necessary} for all $\varepsilon$-order focus values 
to vanish since $V_{i1} \! =\! 0, \ 0 \! \le \! i \! \le \! 11$ under this 
condition. To prove {\it sufficiency},  
under the critical condition ${\rm S_{1c}^8}$, 
we use \eqref{Eqn24} to obtain 
the following $\varepsilon$-order approximation of the first integral, 
\begin{equation}\label{Eqn32}
H_1(x,y,\varepsilon)=\frac{f_1+\varepsilon f_{11}}{f_2+\varepsilon f_{21}},
\end{equation}
where $f_1$ and $f_2$ are given in \eqref{Eqn13}, 
$f_{11}=a_{121}r_{1}+a_{211}r_{2}$ and
$f_{21}=a_{121}r_{3}+a_{211}r_{4}$ with
\begin{equation}\label{Eqn33}
\begin{array}{rl}
r_{1} =& \!\!\!\! -\frac{1}{48}\big[a^2(3a^2+4)(5+2y+2x^2+x^4)
+220 -192ax + 280y\\[0.5ex]
&\!\!\!\! \qquad +120x^2 -64y^2 +128ax^3 +64x^2y +76x^4 \big],\\[1ex]
r_{2} =&\!\!\!\! -\frac{1}{8}a(5a^2-4) +5x -\frac{1}{8}a(a^2-4)(2y+2x^2+x^4)
 +2xy -2x^3,\\[1ex]
r_{3} =&\!\!\!\! \frac{1}{192} \big[a^2(3a^2 \!+\! 4)(4a \!-\! 15x 
\!+\! 10xy \!+\! 10x^3) +304a+16a^3-180x\\[1ex]
& \!\!\!\! \qquad -\, 40 \big( 16ay \!-\! 8ax^2 \!+\! 5xy \!-\! 23x^3
\!-\! 8xy^2 \!+\! 16ax^4 \!+\! 8x^3y \big) \big],\\[1ex]
r_{4}=&\!\!\!\!\frac{a^2}{8}(a^2 \!-\! 1)
\!-\! a(\frac{15}{32}a^2 \!+\! \frac{5}{8})x 
\!+\! \frac{5}{2}x^2(\frac{3}{2} \!+\! y \!-\! x^2) 
\!+\! \frac{5}{16}a(a^2 \!-\! 4)x(y \!+\! x^2). 
\end{array}
\end{equation} 
This implies that setting 
the first $10$ focus values $V_{i1}=0$, $i=0,\cdots,9$
yields $ V_{i1} \!=\!0$ for all $i \! \ge \! 10$.
Moreover, due to that all $V_{i1}$ are linear in  
all parameters $a_{ij1}$ and $b_{ij1}$,   
by Theorem \ref{Thm1} at most 9 
small limit cycles can be obtained for this case. 
The proof is complete. 
\end{proof}

Now suppose the condition ${\rm S_{1c}^8}$ holds and so all 
$\varepsilon$-order 
focus values vanish, we then need to use the $\varepsilon^2$-order focus 
values to study bifurcation of limit cycles. 
With an almost exact same procedure as that used in 
the $\varepsilon$-order analysis, 
we can find a solution $S_2^8$ such that $W_2^8=0$, and then 
\begin{equation}\label{Eqn34} 
\begin{array}{ll} 
V_{92} = G_1 \, a_{032}, \quad 
V_{102} = G_2 \, a_{032}, \quad 
V_{112} = G_3 \, a_{032}, \quad \det_2^{8} = F_1 \ne 0,  
\end{array} 
\end{equation} 
where $F_1$ and $G_i$'s are given in \eqref{Eqn27}. 
Note that the above equations are exactly the same as those given in 
\eqref{Eqn29} and \eqref{Eqn30}, if we replace $k \!=\! 1$ by $k \!=\! 2$ 
in \eqref{Eqn29} and \eqref{Eqn30}. 
This clearly shows that there can exist $9$ limit cycles around the 
equilibrium ${\rm E_0}$ when all $\varepsilon$-order focus values vanish. 
It is also noted that all $V_{i2}$ are linear polynomials 
in $a_{ij2}$ and $b_{ij2}$. 

Similarly, we see that setting $a_{032} \!=\! 0$ in \eqref{Eqn34} yields 
$V_{92} \!=\! V_{102} \!=\!  V_{112} \!=\! 0$, 
implying that the solution $S_2^8$ with $a_{032} \!=\! 0$
and $b_{012}$ given in \eqref{Eqn22} defines 
a necessary condition for all $\varepsilon^2$-order focus values to 
vanish. This critical condition is given below: 
\begin{equation}\label{Eqn35}
\hspace*{-0.20in} {\rm S_{2c}^8\!:} \, \left\{ \!\! 
\begin{array}{rl} 
b_{012} = \!\!\!\! & \textstyle\frac{9}{64} a ^4 a_{211}^2
- \frac{9}{8} a^3 a_{212}  + C_1 a_{122} + C_8 a_{121}a_{211}
+ C_9 a_{121}^2 , \\[1.0ex] 
a_{032} = \!\!\!\! & 0, \ \ 
b_{032} = \textstyle\frac{8}{3}\, a_{122} +5\, a_{121}^2 , \ \ 
b_{212} = \textstyle\frac{a}{2} \, a_{121} (5 a\, a_{121}-9 a_{211}), \\[1.0ex] 
b_{122} = \!\!\!\! & \textstyle\frac{7}{2}\, a_{212} 
- \frac{1}{4}\, a_{121} (31 a\, a_{121}-45 a_{211}) , \\[1.0ex] 
b_{102} = \!\!\!\! & C_2\, a_{122} + C_3 \, a_{212} 
+ C_{10} \, a_{121}^2 + C_{11} \, a_{211}^2 + C_{12} \, a_{121} a_{211}, \\[1.0ex] 
b_{202} = \!\!\!\! & C_4\, a_{122} + C_5 \, a_{212} 
+  C_{13} \, a_{121}^2 + C_{14} \, a_{211}^2
+ C_{15} \, a_{121} a_{211}, \hspace*{-0.15in} \\[1.0ex]   
b_{112} = \!\!\!\! & 9 a\, a_{122}+ \textstyle\frac{9}{2}\, a_{212} 
       + \frac{9a^3}{32} a_{211}^2 
       + C_{16} \, a_{121}^2 + C_{17} \, a_{121} a_{211}, \\[1.0ex] 
b_{022} = \!\!\!\! & -\,6 a_{122} + C_{18} \, a_{121}^2 
+ C_{19} \, a_{121} a_{211}, \\[1.0ex] 
b_{302} = \!\!\!\! & C_6 \, a_{122} + C_7 \, a_{212} + C_{20} \, a_{121}^2 
+ C_{21} \, a_{211}^2 + C_{22} \, a_{211} a_{121},
\end{array} \hspace*{-0.20in}  
\right.    
\end{equation} 
where $C_i$'s are given in Appendix A.  

We have the following theorem.

\begin{theorem}\label{Thm3}
Assume ${\rm S_{1c}^8}$ holds.
The equilibrium ${\rm E_0}$ of system \eqref{Eqn11}
is a center up to $\varepsilon^2$-order, 
if and only if ${\rm S_{2c}^8}$ holds.
Furthermore, there exist at most $9$ small limit cycles
around ${\rm E_0}$ for all parameters $a_{ij2}$ 
and $b_{ij2}$, and $9$ small limit cycles exist for some parameter values 
near ${\rm S_{2c}^8}$.  
\end{theorem}

\begin{proof}
Similarly, we only need to prove {\it sufficiency}.
With ${\rm S_{1c}^8}$ and ${\rm S_{2c}^8}$ holding,
we can use \eqref{Eqn24} to find 
the following $\varepsilon^2$-order approximation of the first integral, 
\begin{equation}\label{Eqn36} 
H_2(x,y,\varepsilon)=\frac{f_1+\varepsilon f_{11}+\varepsilon^2 f_{12}}
{f_2+\varepsilon f_{21}+\varepsilon^2f_{22}}, 
\end{equation}
where $f_{11}$ and $f_{21}$ are given in $H_1(x,y,\varepsilon)$ (see 
Eq.~\ref{Eqn32}), and 
$$ 
\begin{array}{ll} 
f_{21}=a_{122}r_{1}+a_{212}r_{2}+a_{121}^2s_1+a_{211}^2s_2+a_{121}a_{211}s_3, 
\\[0.5ex] 
f_{22}=a_{122}r_{3}+a_{212}r_{4}+a_{121}^2s_4+a_{211}^2s_5+a_{121}a_{211}s_6, 
\end{array} 
$$ 
in which $r_i, \, i=1,2,3,4$ are given in \eqref{Eqn33}, 
and $s_i,\, i=1,2, \dots, 8$, are listed in Appendix B. 

The existence of $9$ small limit cycles is easily seen from 
$V_{92} \! \ne \! 0$ and $\det_2^8 \! \ne \! 0$ when $a_{032} \! \ne \! 0$ 
under the critical condition ${\rm S_{2c}^8}$. 
On the other hand, the above results show that setting 
$V_{i2}$, $0\le i\le 9$ results in 
$V_{i2} \! = \! 0$ for all $ i \! \ge \! 10$.
Further, all $V_{i2}$'s are linear in $a_{ij2}$ and $b_{ij2}$, 
and ${\rm S_{2c}^8}$ is the unique solution of $V_{i2}=0$, $0\le i\le 9$.
Then by Theorem \ref{Thm1}, 
at most $9$ small limit cycles can be obtained around ${\rm E_0}$
for all parameters $a_{ij2}$ and $b_{ij2}$.
\end{proof}

\subsection{$\varepsilon^3$-order analysis} 

In this section, we assume 
the critical condition $\{\rm S_{1c}^8, S_{2c}^8 \}$, 
which stands for that both 
the critical conditions ${\rm S_{1c}^8}$ and ${\rm S_{2c}^8 }$ hold, 
under which all $\varepsilon$- and $\varepsilon^2$-order focus values vanish. 
Thus, we use $\varepsilon^3$-order focus values $V_{i3}$ to 
study bifurcation of limit cycles around the equilibrium 
${\rm E_0}$. With a similar procedure, but for this order, 
we solve $9$ equations  $W_3^9=0$ to obtain the solution $S_3^9$ 
for which  
\begin{equation}\label{Eqn37} 
\begin{array}{ll} 
V_{103} = G_4 \, a_{121}^3, \quad 
V_{113} = G_5 \, a_{121}^3, \quad 
V_{123} = G_6 \, a_{121}^3, \quad 
\det_3^{9} = F_2 \ne 0,  
\end{array} 
\end{equation} 
where $F_2$ and $G_i$'s are given in \eqref{Eqn27}. 
Note that for this order, there is one more independent coefficient
$a_{033}$ in $S_3^9$ for solving $W_3^9 =0$, 
compared to the solutions $S_1^8$ and $S_2^8$ which have only 
$8$ independent coefficients to be used for solving the first 
$8$ focus value equations. 
The equations in \eqref{Eqn37} show that when all 
$\varepsilon$- and $\varepsilon^2$-order focus values vanish, 
the $\varepsilon^3$-order focus values can have solutions such that 
$ V_{i3}=0, \, i = 0, 1, \cdots, 9 $ but $ V_{103} \ne 0$, as well 
as $ \det_3^{9} \ne 0$, implying that $10$ small limit cycles 
can bifurcate from the equilibrium ${\rm E_0}$. 

Setting $a_{121} = 0$ in \eqref{Eqn37}, 
we have $V_{103} = V_{113} = V_{123} =0$, implying that 
under the solution $S_3^9$ with $a_{121}=0$ and 
$b_{013} $ given in \eqref{Eqn22}, the equilibrium ${\rm E_0}$ 
might be a center up to $\varepsilon^3$ order. 
This critical condition is given by 
\begin{equation}\label{Eqn38} 
\hspace*{-0.20in} {\rm S_{3c}^9 }\!: \left\{\!\!  
\begin{array}{rl}
b_{013} = \!\!\!\! 
&  \frac{9}{32} a^4 a_{211} a_{212} 
- \frac{9}{8} a^3 a_{213} + C_1 a_{123} + C_8 a_{122} a_{211} 
+ C_{23} a_{211}^3, \\[1.0ex] 
a_{121} = \!\!\!\! & a_{033} = 0, \quad 
b_{033} = \frac{8}{3} a_{123}, \quad 
b_{023} = -\,6 a_{123} + C_{19} a_{122} a_{211}  
\\[1.0ex] 
b_{213} = \!\!\!\! & -\,\frac{9a}{16} a_{211} (8 a_{122} \!+\! a_{211}^2), 
\quad  b_{123} = \frac{7}{2} a_{213} 
\!+\! \frac{45}{32} a_{211} (8 a_{122} \!+\! a_{211}^2 ), \\[1.0ex] 
b_{113} = \!\!\!\! & 9 a\, a_{123} + \frac{9}{2} a_{213} 
        + a_{211} \big[ \frac{9}{16} a^3 a_{212} + C_{17} a_{122} 
+ C_{24} a_{211}^2 \big] \\[1.0ex] 
b_{103} = \!\!\!\! & C_2 a_{123} + C_3 a_{213}  
+ a_{211} \big[ 2 C_{11}  a_{212} +C_{12} a_{122} 
+ C_{25} a_{211}^2 \big], \\[1.0ex] 
b_{203} = \!\!\!\! & C_4 a_{123} + C_5 a_{213}
+ a_{211} \big[ 2 C_{14} a_{212} +C_{15} a_{122} 
+ C_{26} a_{211}^2 \big]   \\[1.0ex] 
b_{303} = \!\!\!\! & C_6 a_{123} + C_7 a_{213} 
+ a_{211} \big[ 2 C_{21} a_{212} + C_{22} a_{122} 
+ C_{27} a_{211}^2 \big], 
\end{array} \hspace*{-0.20in}  
\right. 
\end{equation}
under which the critical conditions ${\rm S_{1c}^8}$ and ${\rm S_{2c}^8}$ 
are simplified. Here, $C_i$'s are given in Appendix A.  

We have the following theorem. 

\begin{theorem}\label{Thm4}
Let ${\{\rm S_{1c}^8, S_{2c}^8}\}$ hold.
The equilibrium ${\rm E_0}$ of system \eqref{Eqn11}
is a center up to $\varepsilon^3$-order, 
if and only if the condition ${\rm S_{3c}^9}$ holds.
Furthermore, there exist $10$ small limit cycles
around ${\rm E_0}$ for some parameter values of $a_{ij3}$ and $b_{ij3}$ 
near the critical value defined by ${\rm S_{3c}^9}$ when $V_{103}\neq 0$. 
\end{theorem}

\begin{proof}
Similarly again, we only need to prove {\it sufficiency}.
Under the condition $\{\rm S_{1c}^8, S_{2c}^8, S_{3c}^9\}$, 
we obtain the following $\varepsilon^3$-order
approximation of first integral, 
\begin{equation}\label{Eqn39} 
H_3(x,y,\varepsilon)=\frac{f_1+\varepsilon a_{211}r_1
+\varepsilon^2 (a_{122}r_1+a_{212}r_2+a_{211}^2s_2)
+\varepsilon^3 f_{31}}
{f_2+\varepsilon a_{211}r_4
+\varepsilon^2 (a_{122}r_3+a_{212}r_4+a_{211}^2s_5)
+\varepsilon^3 f_{32}}, 
\end{equation}
where 
\begin{equation*}
\begin{split}
&f_{31}=a_{123}r_{1}+a_{213}r_{2}
+a_{211}(a_{122}t_1+a_{212}t_2+a_{211}^2t_3),\\
&f_{32}=a_{123}r_{3}+a_{213}r_{4}
+a_{211}(a_{122}t_4+a_{212}t_5+a_{211}^2t_6),
\end{split}
\end{equation*}
in which $r_i,\, i=1,2,3,4$ are given in \eqref{Eqn33}, 
and $s_2$, $s_5$ and $t_i,\, i=1,2, \dots, 6$ are listed in Appendix B. 
This implies that setting 
$V_{i3}=0$, $0\le i\le 10$ yields $V_{i3} \! =\! 0$ for all $ i \! \ge \! 11$. 
Then, there exist at most $10$ small limit cycles 
for this case. On the other hand, $10$ small limit cycles exist 
since when $a_{121} \! \ne \! 0$, $V_{101} \! \ne \! 0$ and 
$ \det_3^9 \! \ne \! 0$. 
\end{proof}

\subsection{$\varepsilon^4$--$\varepsilon^6$-order analysis} 

The analyses for $\varepsilon^4$-,  $\varepsilon^5$- 
and $\varepsilon^6$-order are  similar to that 
of $\varepsilon^1$-, $\varepsilon^2$- and $\varepsilon^3$-order, 
respectively. 

Let $\{ {\rm S_{1c}^8, S_{2c}^8, S_{3c}^9} \}$ hold.
Following the same procedure used in the previous sections, we 
can solve the equations $W_4^8=0$ to obtain a solution $S_4^8$ such that 
\begin{equation}\label{Eqn40} 
V_{94} = G_1 \, a_{034}, \ \  
V_{104} = G_2 \, a_{034}, \ \  
V_{114} = G_3 \, a_{034}, \ \ {\det}_{4}^{8} = F_1 \ne 0,   
\end{equation} 
which has the exactly same form of the equations as those given in 
\eqref{Eqn30} and \eqref{Eqn34}, implying that perturbing the 
$\varepsilon^4$-order focus values from the solution $S_4^8$
and $b_{014}$ (see \eqref{Eqn22}) 
can yield $9$ limit cycles around the equilibrium ${\rm E_0}$. 
Similarly, the solution $S_4^8$ and $b_{014}$ with $a_{034} \! =\! 0$  
yields a critical condition ${\rm S_{4c}^8} $, under which 
the equilibrium ${\rm E_0}$ is a center up to $\varepsilon^4$ order.  

Then let $\{ {\rm S_{1c}^8, S_{2c}^8, S_{3c}^9,S_{4c}^8} \}$ hold.
In the same line, we can solve the equations 
$W_5^8=0$ to obtain a solution $S_5^8$ such that 
\begin{equation}\label{Eqn41}  
\begin{array}{rl} 
V_{95} = \!\!\! & G_1 \,A_{035}, \ \ 
V_{105} = G_2\,  A_{035}, \ \ 
V_{115} = G_3\,  A_{035}, \ \ 
{\det}_5^{8} = F_1 \ne 0,  
\end{array}
\end{equation}
where 
\begin{equation}\label{Eqn42} 
A_{035} = a_{035}
+ \textstyle\frac{1}{48}\, a_{122}\, a_{211} 
\, (140 \, a_{122}+35 \, a_{211}^2). 
\end{equation}
This shows that perturbing the $\varepsilon^5$-order 
focus values near the solution $S_5^8$ 
and $b_{015}$ given in \eqref{Eqn22} 
can also yield $9$ limit cycles around the equilibrium ${\rm E_0}$. 
It is easy to see that the solution of $A_{035} = 0$, 
\begin{equation}\label{Eqn43}  
a_{035} = -\, \textstyle\frac{35}{48}\, a_{122}\, a_{211} 
\, (4 \, a_{122}+ a_{211}^2), 
\end{equation} 
yields $V_{95}= V_{105} = V_{115} = 0$.
Now, we combine the solution $S_5^8$, $b_{015}$ and $a_{035}$ 
to obtain the critical condition ${\rm S_{5c}^8}$, under which
the equilibrium ${\rm E_0}$ becomes 
a center up to $\varepsilon^5$ order.  

The lengthy critical conditions ${\rm S_{4c}^8 }$ and ${\rm S_{5c}^8 }$ are 
omitted here for brevity. 
Summarizing the above results leads to the following theorem.

\begin{theorem}\label{Thm5}
System \eqref{Eqn11} can have maximal $9$ limit cycles 
around the equilibrium ${\rm E_0}$ under 
the condition $\{\rm S_{1c}^8, S_{2c}^8, S_{3c}^9\}$ 
by perturbing the $\varepsilon^4$-order focus values
around the critical value ${\rm S_{4c}^8}$; 
and under the critical condition 
$\{\rm S_{1c}^8, S_{2c}^8, S_{3c}^9, S_{4c}^8\}$ 
by perturbing the $\varepsilon^5$-order 
focus values near the critical point ${\rm S_{5c}^8}$. 
The equilibrium ${\rm E_0}$ becomes a center up to $\varepsilon^4$ order  
under the condition $\{\rm S_{1c}^8, S_{2c}^8, S_{3c}^9, S_{4c}^8\}$, 
and a center up to $\varepsilon^5$ order under the condition 
$\{\rm S_{1c}^8, S_{2c}^8, S_{3c}^9, S_{4c}^8, S_{5c}^8\}$.  
\end{theorem}

\begin{rmk}\label{Rem2} 
The proof for the center conditions in Theorem~\ref{Thm5} 
is similar to that in proving Theorems~\ref{Thm2}, \ref{Thm3}
and \ref{Thm4} by finding 
the $\varepsilon^4$-order and $\varepsilon^5$-order approximations 
of the first integrals. This is the major and tedious part. 
For higher-order analysis, the proofs are similar. 
We omit the detailed proofs in the following analysis for brevity. 
\end{rmk} 

Next, suppose the condition 
$\{\rm S_{1c}^8, S_{2c}^8, S_{3c}^9, S_{4c}^8, S_{5c}^8\}$ 
is satisfied, then all 
$\varepsilon^k,\, k=1,2, \dots, 5$, order focus values vanish. 
Following a similar analysis as that for $\varepsilon^3$ order, we 
solve the equations $W_6^9=0$ 
to obtain a solution $S_6^9 $ such that 
\begin{equation}\label{Eqn44}  
\begin{array}{rl} 
V_{106} = \!\!\! & G_4 \, 
a_{122}^2 \big( a_{122} + \frac{9}{8}\, a_{211}^2 \big) , \\[1.0ex]   
V_{116} = \!\!\! & G_5 \, 
a_{122}^2 \big( a_{122} + \frac{9}{8}\, a_{211}^2 \big) , \\[1.0ex]   
V_{126} = \!\!\! & G_6 \, 
a_{122}^2 \big( a_{122} + \frac{9}{8}\, a_{211}^2 \big) , \quad 
{\det}_6^{9} = F_2 \ne 0,   
\end{array} 
\end{equation} 
which indeed shows the existence of $10$ limit cycles around the 
equilibrium ${\rm E_0}$, generated from perturbing the $\varepsilon^6$-order
focus values near the solution $S_6^9$ 
under the condition 
$\{\rm S_{1c}^8, S_{2c}^8, S_{3c}^9, S_{4c}^8, S_{5c}^8\}$. 
Moreover, when $ a_{122} \!=\! -\frac{9}{8}\, a_{211}^2 $ or $ a_{122} \!=\!0$, 
we have $ V_{106} \!=\! V_{116} \!=\! V_{126} \!=\! 0$, 
indicating that the solution $S_6^9$ with either $ a_{122} \!=\! 
-\frac{9}{8}\, a_{211}^2 $ or $ a_{122} \!=\! 0$, plus 
$b_{016}$ given by \eqref{Eqn22}$|_{k=6}$, yields a 
critical condition ${\rm S_{6c}^{9a}}$ (corresponding to 
the former) or ${\rm S_{6c}^{9b}}$ (corresponding to the latter) 
under which all $\varepsilon^6$-order focus values vanish. 
Thus, under the critical condition 
$\{\rm S_{1c}^8, S_{2c}^8, S_{3c}^9, S_{4c}^8, S_{5c}^8, S_{6c}^9\}$
(${\rm S_{6c}^9 } $ 
equals either ${\rm S_{6c}^{9a}}$ or ${\rm S_{6c}^{9b}}$), 
the equilibrium ${\rm E_0}$ becomes a center 
up to $\varepsilon^6$ order. 

We have the following theorem for this order. 

\begin{theorem}\label{Thm6}
System \eqref{Eqn11} can have maximal $10$ limit cycles 
bifurcating from the equilibrium ${\rm E_0}$ under 
the condition 
$\{\rm S_{1c}^8, S_{2c}^8, S_{3c}^9, S_{4c}^8, S_{5c}^8\}$ 
by perturbing the $\varepsilon^6$-order 
focus values near the critical point ${\rm S_{6c}^{9a}}$ 
or ${\rm S_{6c}^{9b}}$. 
Further, the equilibrium ${\rm E_0}$ becomes a center 
up to $\varepsilon^6$ order under the condition 
$\{\rm S_{1c}^8, S_{2c}^8, S_{3c}^9, S_{4c}^8$, $S_{5c}^8, S_{6c}^9 \}$, 
for which all $\varepsilon^k$-order {\rm (}$k=1,2, \dots, 6${\rm )} 
focus values vanish. 
\end{theorem}

Suppose the condition 
$\{\rm S_{1c}^8, S_{2c}^8, S_{3c}^9, S_{4c}^8, S_{5c}^8, S_{6c}^9 \}$ holds. 
Then, all the $\varepsilon^k$-order 
($k=1,2, \dots, 6$) focus values vanish. We have two cases 
for higher-order analysis, defined as 
\begin{equation}\label{Eqn45}
\begin{array}{ll}  
{\rm Case \ (A) \ \ 
\{ S_{1c}^8, S_{2c}^8, S_{3c}^9, S_{4c}^8, S_{5c}^8, S_{6c}^{9a}\}}, \\[1.0ex] 
{\rm Case \ (B) \ \ 
\{S_{1c}^8, S_{2c}^8, S_{3c}^9, S_{4c}^8, S_{5c}^8, S_{6c}^{9b}\}}. 
\end{array} 
\end{equation}

\subsection{Higher-order analysis for Case (A)} 

First we consider Case (A), under which 
we will show that $11$ limit cycles can bifurcate from the equilibrium 
${\rm E_0}$ based on the $\varepsilon^7$-order focus values. 

\subsubsection{$\varepsilon^7$-order analysis} 

Under the condition (A) defined in \eqref{Eqn45} 
with $ a_{122} = - \frac{9}{8} a_{211}^2$, we obtain 
\begin{equation}\label{Eqn46} 
{\det}_7^{10} = F_3\, a_{211}^4 , \quad 
{\det}_7^{11} = F_4\, a_{211}^{10}, 
\end{equation}  
which shows that ${\det}_7^{10} \ne 0 $ and 
$ {\det}_7^{11} \ne 0 $ when $a_{211} \ne 0$, implying that 
we may have solutions such that the first ten focus values vanish
but $V_{117} \ne 0$ and so $11$ small limit cycles may be obtained.   
Indeed, we can solve the first ten focus values 
equations: $W_k^{10} = 0$ to obtain a solution $S_7^{10}$ such that  
\begin{equation}\label{Eqn47} 
V_{117} = G_7 \, a_{211}^7, \ \  
V_{127} = G_8 \, a_{211}^7, \ \  
V_{137} = G_9 \, a_{211}^7, 
\end{equation} 
which clearly shows that $V_{117} \ne 0 $ if $a_{211} \ne 0$. 
In addition, due to ${\det}_7^{10} \ne 0$ when $a_{211} \ne 0$,
implying that $11$ small limit cycles exist. 

Letting $a_{211} \!=\! 0 $, we have $V_{117} \!=\! V_{127} 
\!=\! V_{137} \!=\! 0$, leading to a critical condition ${\rm S_{7c}^{10}}$,
defined by 

\begin{equation}\label{Eqn48}
{\rm S_{7c}^{10}}\!: \left\{\! 
\begin{array}{rl}
b_{017} = \!\!\! & 
C_1 a_{127} 
\!-\! \frac{9}{8} a^3 a_{217}  
\!+\! C_8 C_{28} \!+\! \frac{9}{32} a^4 C_{29} 
\!+\! C_{23} C_{30} ,  \\[1.0ex] 
a_{211} = \!\!\! & a_{123} = a_{037} = 0, \hspace*{0.27in} b_{037} 
= \frac{8}{3} \, a_{127}, \\[1.0ex] 
b_{027} = \!\!\! & -6 a_{127}  +  C_{19} C_{28}, \quad 
b_{217} = -\frac{9a}{16} (8 C_{28} + C_{30} ), 
\\[1.0ex] 
b_{127} = \!\!\! & \frac{7}{2} a_{217} 
+ \frac{45}{32} ( 8 C_{28} + C_{30} ), \\[1.0ex] 
b_{107} = \!\!\! & C_2 a_{127} + C_3 a_{217} 
+ 2 C_{11} C_{29} + C_{12} C_{28} + C_{25} C_{30}, \\[1.0ex] 
b_{207} = \!\!\! & C_4 a_{127} + C_5 a_{217} 
+ 2 C_{14} C_{29} + C_{15} C_{28} 
        + C_{26} C_{30} , \\[1.0ex]  
b_{117} = \!\!\! & 9 a\, a_{127}+ \frac{9}{2} \, a_{217} 
+ \frac{9 a^3}{16} C_{29} + C_{17} C_{28} + C_{24} C_{30} , \\[1.0ex] 
b_{307} = \!\!\! & C_6 a_{127} + C_7 a_{217}  
+ 2 C_{21} C_{29} + C_{22} C_{28} 
        + C_{27} C_{30}, 
\end{array} 
\right. 
\end{equation} 
where $C_i$'s are given in Appendix A. 

We have the following result. 

\begin{theorem}\label{Thm7}
Let ${\rm\{ S_{1c}^8, 
S_{2c}^8, S_{3c}^9, S_{4c}^8, S_{5c}^8, S_{6c}^{9a}\}}$ hold.
The equilibrium ${\rm E_0}$ of \eqref{Eqn11} becomes a center 
up to $\varepsilon^7$ order under 
${\rm S_{7c}^{10}}$ for which all $\varepsilon^7$-order 
focus values vanish. Furthermore, there exist 
$11$ small limit cycles around ${\rm E_0}$ 
for parameter values of $a_{ij7}$ and $b_{ij7}$ near the critical point 
${\rm S_{7c}^{10}}$. 
\end{theorem}

\subsubsection{Higher-order analysis}

For higher-order analysis ($k \ge 8$), 
we first briefly list the results for a few orders to see the 
patterns and then summarize the results in a table for higher orders.  

The analysis on $\varepsilon^k$ ($k=8,9,10,11$) orders 
show the same pattern, giving 
$9$ limit cycles for each order, as follows: 
\begin{equation}\label{Eqn49} 
\hspace*{-0.20in} \begin{array}{cll} 
\begin{array}{c}
{\rm Order} \ k\!: \\ 
(k\!=\!8,9,10,11) 
\end{array} & \!\!\! \{S_k^8,W_k^8\},  
\!\!\! & \left\{\!\! \begin{array}{rl} 
V_{9k} = \!\!\!\! & G_1 \, A_{03k}, \ \  
V_{10k} = G_2 \, A_{03k}, \\ [0.7ex]  
V_{11k} = \!\!\!\! & G_3 \, A_{03k}, \ \ 
{\det}_{k}^{8} = F_1 ,
\end{array} 
\right.
\end{array}  
\end{equation} 
where $\{S_k^m,W_k^m\}$ denotes the solution $S_k^m$ solved from 
$W_k^m=0$, and 
$$ 
\begin{array}{rl} 
A_{038} = \!\!\! & a_{038}, \qquad A_{039} = a_{039}, \\[0.7ex] 
A_{0310} = \!\!\! & a_{0310} + \textstyle\frac{35}{48} 
\,  a_{124}\, a_{212} \, (4 \, a_{124}+ a_{212}^2), \\[1.0ex] 
A_{0311} = \!\!\! & a_{0311} 
+ \textstyle\frac{35}{48} \big[ a_{125} a_{212} (8 a_{124}+a_{212}^2) 
+ a_{124} a_{213} (4 a_{124} +3 a_{212}^2) \big]. 
\end{array}
$$ 
This clearly shows that for each order of $k=8,9,10,11$, one can 
solve $A_{03k}=0$ to get a unique solution for $ a_{03k}$ 
under which (together with the solutions $S_k^m$ 
and $b_{01k}$ obtained in the previous 
orders and the current order)
the equilibrium ${\rm E_0}$ becomes a center up to that order.  

When the equilibrium ${\rm E_0}$ is a center up to $11$th order, 
as given in \eqref{Eqn49}
we obtain the following result for order $12$:  
\begin{equation}\label{Eqn50} 
\hspace*{-0.15in} \begin{array}{cll} 
{\rm Order} \ 12\!:  
\!\! & \{S_{12}^9,W_{12}^9\},  
\!\!\! & \left\{\!\! \begin{array}{rl} 
V_{1012} = \!\!\!\! & G_4\, a_{124}^2 \big(a_{124} 
\!+\! \textstyle\frac{9}{8} a_{212}^2 \big), \\ [1.0ex]  
V_{1112} = \!\!\!\! & G_5\, a_{124}^2 \big(a_{124}
\!+\! \textstyle\frac{9}{8} a_{212}^2 \big) , \\ [1.0ex] 
V_{1212} = \!\!\!\! & G_6\, a_{124}^2 \big(a_{124}
\!+\! \textstyle\frac{9}{8} a_{212}^2 \big), \ \ 
{\det}_{12}^{9} = F_2 ,  
\end{array} 
\right.  
\end{array}
\end{equation} 
which has the exactly the same pattern as order $6$, shown in 
\eqref{Eqn44}, indicating that $10$ limit cycles can be obtained from 
this order, and there are two solutions from the 
equations $V_{1012} \! =\! V_{1112} \!=\! V_{1212} \! =\! 0$: 
$a_{124} \!=\! -\frac{9}{8} a_{212}^2 $ and $a_{124} \!=\! 0$, which are 
again similar to that as 
in order $6$. When $a_{124} = 0$, it will be shown 
in Section 4.7 that 
it yields the same pattern as that for Case (B) in higher orders.    
So in this section, we choose $ a_{124} \!=\! -\frac{9}{8} a_{212}^2$, 
like we chose $a_{122} \!=\! -\frac{9}{8} a_{212}^2$ in order 
$6$ to obtain the center condition. 

Let $ a_{124} \!=\! -\frac{9}{8} a_{212}^2$, under which (together with 
the solutions obtained from previous orders and this order) the equilibrium 
${\rm E_0}$ becomes a center up to $\varepsilon^{12}$ order. 
Then, we have the result for $\varepsilon^{13}$ order: 
\begin{equation}\label{Eqn51} 
\hspace*{-0.00in} \begin{array}{cll} 
{\rm Order} \ 13\!: 
& \!\!\! \{S_{13}^9,W_{13}^9\},  
\!\!\! & \left\{\!\! \begin{array}{rl} 
V_{1013} = \!\!\!\! & 
\textstyle\frac{81}{64} G_4\, a_{212}^4 
\big(a_{125}+ \textstyle\frac{9}{4} \,a_{212} a_{213} \big),  \\[1.0ex] 
V_{1113} = \!\!\!\! & \textstyle\frac{81}{64} G_5\, a_{212}^4 
\big(a_{125}+ \textstyle\frac{9}{4} \,a_{212} a_{213} \big),  \\[1.0ex] 
V_{1213} = \!\!\!\! & 
\textstyle\frac{81}{64} G_6\, a_{212}^4 
\big(a_{125}+ \textstyle\frac{9}{4} \,a_{212} a_{213} \big), \ \ 
{\det}_{13}^{9} = F_2 ,  
\end{array} 
\right.  
\end{array} 
\end{equation}  
which shows that perturbing $\varepsilon^{13}$-order focus values 
can also yield $10$ small limit cycles around the equilibrium ${\rm E_0}$. 
It can be seen from \eqref{Eqn51} that either $a_{212} \!=\! 0$ or 
$ a_{125} \!=\! -\frac{9}{4} a_{212} a_{213} $ leads to 
the equilibrium ${\rm E_0}$ being a center. However, it can be shown 
that setting $ a_{212} \!=\! 0$ at this order will not yield $11$ 
small limit cycles at the next order though it will resume the same pattern 
at higher orders. 

So let $ a_{125} \!=\! -\frac{9}{4} a_{212} a_{213} $. Then, we obtain 
the following result for $\varepsilon^{14}$ order: 
\begin{equation}\label{Eqn52} 
\hspace*{-0.20in} \begin{array}{cll} 
{\rm Order \, 14\!:} & \{S_{14}^{10},W_{14}^{10}\},  
& \left\{\!\! \begin{array}{rl} 
V_{1114} = \!\!\!\! & G_7 \, a_{212}^7, \ \ 
V_{1214} = G_8\,  a_{212}^7 , \\ [0.7ex] 
V_{1314} = \!\!\!\! & G_9\,  a_{212}^7, \ \ 
{\det}_{14}^{10} = F_3 \ne 0,  
\end{array} 
\right.  
\end{array} 
\end{equation} 
which shows that perturbing $\varepsilon^{14}$-order focus values
can yield $11$ limit cycles around the equilibrium ${\rm E_0}$, 
and setting $a_{212} = 0$ leads to ${\rm E_0}$ being a center 
up to $\varepsilon^{14}$ order. 
It has been noted that choosing $a_{212} \!=\! 0$ at order $13$ or $14$ makes
differences. More precisely, as shown in Table~\ref{table1}, if taking 
$ a_{125} \!=\! -\frac{9}{4} a_{212} a_{213} $ at 
order $13$, we have small limit cycles $11, \, 9, \, 9, \, 9, \,9$ 
for the orders $14$--$18$; while if taking $ a_{212} \!=\! 0$ at order $13$, 
then the limit cycles obtained for the orders $14$--$18$  
are $9,\, 10, \, 9, \, 9, \, 10$, and then the two different choices 
merge into the same pattern from order $19$. Note that the choice 
$a_{212} \!=\! 0$ at order $13$ does not yield $11$ small limit cycles 
at order $14$, but 
gives two more $10$ small limit cycles at orders $15$ and $18$. 
However, it returns to the general pattern at order $19$.
So we treat $a_{212} \!=\! 0$ as a special case of the 
case $ a_{125} \!=\! -\frac{9}{4} a_{212} a_{213} $. 

Summarizing the above results we have the following pattern: $11$ 
limit cycles are obtained from $\varepsilon^{7}$ order, then $9$ limit cycles 
from four consecutive $\varepsilon^{k}$ orders ($k=8,9,10,11$), and then 
$10$ limit cycles from two consecutive $\varepsilon^{k}$ orders ($k=12,13$), 
and finally return to $11$ limit cycles at $\varepsilon^{14}$ order.    
This pattern, starting from order $8$, four 
$9$ limit cycles, followed by two $10$ limit cycles, and then 
$11$ limit cycle, has been verified up to $\varepsilon^{35}$ order. 
We call this as $9^4$-$10^2$-$11^1$ 
{\it generic} pattern, and the corresponding 
solution (or center condition) is called {\it generic solution} 
(or generic center condition). 
By {\it generic} we mean that 
one should always choose a non-zero solution (if it exists) 
when one solves the center conditions at each order.  
Other types of solutions are called {\it non-generic}. 
For example, as discussed above, if choosing the non-generic solution 
$a_{212} \!=\! 0$ at order $13$, then $11$ limit cycles will be missed  
at order $14$ but the solution procedure will return to the 
generic $9^4$-$10^2$-$11^1$ pattern at order $19$.  
However, it should be noted that a non-generic solution in Case (A) 
does not always lead to the generic $9^4$-$10^2$-$11^1$ pattern. 
For instance, choosing the non-generic solution 
$a_{124} \!=\! 0$ at order $12$ will generate solutions 
in the form of generic patter of Case (B) at a higher order, 
as shown in the next section.

\begin{rmk}\label{Rem3}
It has been observed from the above analysis, the values of the 
parameter $a$ in the Hamiltonian function does not affect the number of 
limit cycles. In other words, $a$ can not be used to increase the 
number of bifurcating limit cycles. Thus, to simply the computations 
in higher order analysis, we set $a=-\,3$ in higher-order 
($k \ge 15$) computations, which greatly simplify the computations. 
\end{rmk} 

We summarize the results of Case (A) in Table~\ref{table1}, 
where $k$ is the order of $\varepsilon^k$ focus values, 
$(S_k^m,W_k^m)$ represents the solution $S_k^m$ solved from 
$ W_k^m \!=\! 0$, and ${\rm LC}$ denotes limit cycles 
around the equilibrium ${\rm E_0}$ obtained by perturbing the 
$\varepsilon^k$-order focus values. 
The ``Condition for Center'' in each row only lists the condition for 
the current row, which assumes that the conditions in the previous 
rows hold. For example, when $k\!=\! 4$, ${\rm S_{4c}^8}$ only gives the 
center condition for $k \!=\! 4$, which should be combined with the 
conditions given in the previous rows: ${\rm S_{1c}^8}$, 
${\rm S_{2c}^8}$ and ${\rm S_{3c}^9}$ to get a complete 
center condition $\{\rm S_{1c}^8,S_{2c}^8,S_{3c}^9,S_{4c}^8\}$. 
Note that the critical condition ${\rm S_{kc}^8}$ contains the solutions 
$S_k^8$, the  $b_{01k}$ given in \eqref{Eqn22} 
and a particular coefficient. For example, ${\rm S_{2c}^8} 
\!=\! \{S_2^8, b_{012}, a_{032} \}$, ${\rm S_{3c}^9} \!=\! \{S_3^9,
 b_{013} , a_{121} \} $, and 
${\rm S_{7c}^{10}} \!=\! \{S_7^{10},
b_{017} , a_{211} \} $, etc. 
The solutions of these key coefficients are given below. 
$$ 
\begin{array}{lll} 
{\rm 9 \ LC:} & k=1,2,4,8,9 & a_{03k}=0  \\  
& k=5  & a_{035}=-\frac{35}{48} a_{122}a_{211} (4 a_{122} + a_{211}^2)  
\\ [0.5ex]  
& k=10 & a_{0310}=-\frac{35}{48} a_{124}a_{212} 
(4 a_{124} + a_{212}^2) \\  [0.5ex] 
& k=11 & a_{0311} = -\frac{35}{48}
\big[ a_{212} a_{125} (8 a_{124}+a_{212}^2) \\ 
& & \qquad \qquad \quad +a_{213} a_{124} (4 a_{124}+3 a_{212}^2) \big]
\\[0.5ex]  
& k=15 & a_{0315}= -\frac{735}{256} a_{213}^5 \\ [0.5ex] 
& k=16 & a_{0316}= \frac{35}{768} a_{213}^3  
(128 a_{127} - 27 a_{214} a_{213}) \\ 
& k=17 & a_{0317} = -\frac{35}{768} a_{213} \big[ 
32 a_{127} (2 a_{127}-3 a_{214} a_{213}) \\ 
& & \qquad \qquad \quad -a_{213}^2 (128 a_{128} \!+\! 54 a_{214}^2
\!-\! 27 a_{215}  a_{213}) \big]  \\  [0.5ex] 
& k=18 & a_{0318} = \frac{35}{6} a_{213}^3 a_{129} 
        -\frac{35}{24} a_{213} a_{128} (4 a_{127}-3 a_{214} a_{213}) \\[0.5ex]
& & \qquad \quad 
- \frac{35}{48} a_{127} \big[ a_{214} (4 a_{127} \!+\! 3 a_{214} a_{213})
\!-\! 6 a_{213}^2 a_{215} \big] \\[0.5ex]  
& & \qquad \quad 
+ \frac{3}{4096} a_{213}^2 \big[ 1120 a_{214} (a_{214}^2+6 a_{215} a_{213})
\\[0.5ex] 
& & \qquad \qquad \qquad \quad +3 a_{213}^2 (2269 a_{213}^2-560 a_{216}) \big] 
  \\ [0.5ex] 
& k=22 & a_{0322}= \frac{35}{768} 
a_{214}^3 (128 a_{1210} \!-\! 486 a_{215}^2 \!-\! 27 a_{216} a_{214})\\[0.5ex]
& k=23 &  a_{0323} = \frac{35}{768} a_{214}^2 \big[ 
a_{214} (128 a_{1211} -27 a_{217} a_{214}) \\ [0.5ex] 
& & \qquad \qquad 
+a_{215} (384 a_{1210} -108 a_{214} a_{216} -198 a_{215}^2) \big] \\ [0.5ex] 
& k=24,25 &    a_{03k} = \cdots  \\ 
& k=29$--$32 & a_{03k} = \cdots  \\ [1.0ex] 
{\rm 10 \ LC:} & k=3 & a_{121} = 0 \\ 
& k=6 & a_{122} =-\frac{9}{8} a_{211}^2 \\ 
& k=12 & a_{124} =-\frac{9}{8} a_{212}^2  \\   
& k=13 & a_{125} =-\frac{9}{4} a_{213} a_{212} \\ 
& k=19 & a_{127} =-\frac{9}{4}\, a_{213}a_{214} \\  
& k=20 & a_{128} =-\frac{9}{8} (a_{214}^2+2 a_{215} a_{213}) \\ 
& k=26 & a_{1210} =-\frac{9}{8} (a_{215}^2+2 a_{216} a_{214}) \\ 
& k=27 & a_{1211} =-\frac{9}{4} (a_{217} a_{214}+a_{216} a_{215})  \\ 
& k=33 & a_{1213} =-\frac{9}{4} (a_{218} a_{215}+a_{217} a_{216})  \\ 
& k=34 & a_{1214} =-\frac{9}{8} (a_{217}^2 +2 a_{219} a_{215} 
+ 2 a_{218} a_{216}) \qquad \qquad \qquad \qquad \\ [1.0ex] 
{\rm 11 \ LC:} & k=7m &  \\  
&  m= 1$--$5 & a_{21m} = 0   
\end{array} 
$$
where `$ \cdots $' 
represents the omitted lengthy expressions for brevity. 
In addition, in Table~\ref{table1}, the blue and red colors denote 
the solutions and center conditions corresponding to the 
$10$ and $11$ small limit cycle, respectively. 

\begin{table}[!h]
\caption{Bifurcation of limit cycles for generic Case (A)}  
\centering 
\vspace*{0.15in} 
\label{table1}
\begin{tabular}{|c|c|c|c|}
\hline
$k$ & ($S_k^m,W_k^m$) & No. of LC & Condition for Center   \\
\hline
1,2 & $(S_k^8,W_k^8)$  & 9 & ${\rm S_{kc}^8}$  \\  
 3  & $(S_3^9,W_3^9)$  & {\color{green}{10}} 
& ${\color{green}{\rm S_{3c}^9}}$  \\  
 4,5 & $(S_k^8,W_k^8)$  & 9 & ${\rm S_{kc}^8}$  \\  
 6  & $(S_6^9,W_6^9)$  & {\color{green}{10}} 
&  ${\color{green}{\rm S_{6c}^9}}$  \\ 
 7 & $(S_7^{10},W_7^{10})$  & {\color{red}{11}} 
& ${\color{red}{\rm S_{7c}^{10}}}$ \\ 
            &  &  &  \vspace*{-0.15in}  \\ \hline  
            &  &  &  \vspace*{-0.15in}  \\ 
8--11 & $(S_k^8,W_k^8)$  & 9 & ${\rm S_{kc}^8}$  \\  
 12,13 & $(S_{12}^9,W_{12}^9)$ & {\color{green}{10}}
& ${\color{green}{\rm S_{kc}^9}} $  \\  
            &  &  &  \vspace*{-0.17in}  \\ 
 14 & $(S_{14}^{10},W_{14}^{10})$  & {\color{red}{11}}
& ${\color{red}{\rm S_{14c}^{10}}}$ \\ 
            &  &  &  \vspace*{-0.15in}  \\ \hline  
            &  &  &  \vspace*{-0.15in}  \\ 
15--18 & $(S_{15}^8,W_{15}^8)$  & 9 & ${\rm S_{kc}^8} $ \\ 
            &  &  &  \vspace*{-0.17in}  \\ 
 19,20 & $(S_{19}^8,W_{19}^9)$  & {\color{green}{10}} 
& ${\color{green}{\rm S_{kc}^9}}  $ \\ 
            &  &  &  \vspace*{-0.17in}  \\ 
 21 & $(S_{21}^{10},W_{21}^{10})$  & {\color{red}{11}}
& ${\color{red}{\rm S_{21c}^{10}}} $ \\ 
            &  &  &  \vspace*{-0.15in}  \\ \hline  
            &  &  &  \vspace*{-0.15in}  \\ 
 22--25 & $(S_{22}^8,W_{22}^8)$  & 9 & ${\rm S_{kc}^8}$ \\ 
 26,27 & $(S_{26}^{9},W_{26}^{9})$  & {\color{green}{10}}
& ${\color{green}{\rm S_{kc}^9}} $ \\ 
            &  &  &  \vspace*{-0.17in}  \\ 
 28 & $(S_{28}^{10},W_{28}^{10})$  & {\color{red}{11}}
& ${\color{red}{\rm S_{28c}^{10}}} $ \\ 
            &  &  &  \vspace*{-0.15in}  \\ \hline  
            &  &  &  \vspace*{-0.15in}  \\ 
29--32 & $(S_{k}^8,W_{k}^8)$  & 9 & $ {\rm S_{kc}^8} $\\ 
 33,34 & $(S_{33}^{9},W_{33}^{9})$  & {\color{green}{10}}
& ${\color{green}{\rm S_{kc}^9}} $ \\ 
            &  &  &  \vspace*{-0.17in}  \\ 
 35 & $(S_{35}^{10},W_{35}^{10})$  & {\color{red}{11}}
& ${\color{red}{\rm S_{35c}^{10}}} $  \vspace*{-0.15in} \\ 
& & & \\ \hline
\end{tabular}
\end{table}

\subsection{Higher-order analysis for Case (B)} 

We now turn to Case (B) for which we choose $a_{122} \!=\! 0$ at 
$\varepsilon^6$ order. 
Thus, the results starting from $\varepsilon^6$ order 
are different from those given in Table~\ref{table1}. 
Now under the condition $a_{122} \!=\! 0$, together with the 
conditions obtained in previous 
orders, the equilibrium ${\rm E_0}$ becomes a center up to 
 $\varepsilon^6$ order. Then for $\varepsilon^7$-order  
focus values we solve $W_7^8 =0$ to obtain $S_7^8$ and then 
$$ 
V_{97} = G_1 \, A_{037}, \ \  
V_{107} = G_2 \, A_{037}, \ \  
V_{117} = G_3 \, A_{037},  \ \ {\det}_7^8 = F_1, 
$$ 
where 
$$
A_{037} = 
a_{037} \!+\! \textstyle\frac{35}{768} a_{211} 
\big[ a (a^2 \!-\! 4) a_{123} a_{211}^3
\!+\! 16 a_{124} a_{211}^2 \!+\! 16 a_{123} (4 a_{123} \!+\! 
3 a_{212} a_{211}) \big], 
$$ 
which shows that for Case (B) only $9$ small limit cycles can be obtained from 
$\varepsilon^7$-order. 
Then, solving $A_{037} \!=\! 0$ gives a unique solution for $a_{037}$, 
under which, together with the conditions obtained in the previous 
orders as well as $S_7^8$ and $b_{017}$, 
the equilibrium ${\rm E_0}$ becomes a center up to 
$\varepsilon^7$ order.  

Next, the $\varepsilon^8$-order analysis shows that $10$ limit cycles 
can be obtained by solving $W_8^9 =0$ to have the solution $S_8^9$, 
under which higher-order focus values become 
\begin{equation}\label{Eqn53} 
V_{108} = \textstyle\frac{9}{8} G_4 a_{211}^2 a_{123}^2, \ \  
V_{118} = \textstyle\frac{9}{8} G_5 a_{211}^2 a_{123}^2, \ \  
V_{128} = \textstyle\frac{9}{8} G_6 a_{211}^2 a_{123}^2, \ \  
{\det}_8^9 = F_2. 
\end{equation} 
This clearly indicates that either $ a_{211} \!=\! 0$ or $a_{123} \!=\! 0$, 
together with $b_{018}$,  
leads to the equilibrium ${\rm E_0}$ being a center up to 
$\varepsilon^8$ order. 
If taking $ a_{211} \!=\! 0$, then we again obtain $10$ small limit cycles 
from $\varepsilon^9$ order by solving $W_9^9 =0$ to obtain the solution 
$S_9^9$ and 
$$ 
V_{109} = G_4\, a_{123}^3, \quad   
V_{119} = G_5\, a_{123}^3, \quad   
V_{129} = G_6\, a_{123}^3, \quad  {\det}_9^9 = F_2. 
$$ 
Thus, for the equilibrium ${\rm E_0}$ being a center up to 
$\varepsilon^9$ order, $a_{123}$ must be taken zero (with 
$b_{019}$), yielding the same result as that generated from Case (A) at 
order $9$ (and so the result at order $8$ also becomes same). 
In other words, choosing the non-generic solution $a_{211} \!=\! 0$ 
at order $8$ makes the higher-order solutions ($k \ge 9$) follow the 
generic pattern of Case (A). 

Now we consider the choice $a_{123} \!=\! 0$ at $\varepsilon^8$ order. 
It can be shown that under this condition only $9$ limit cycles 
exist for $\varepsilon^9$ order. Then for the $\varepsilon^{10}$ order, 
we solve $W_{10}^9=0$ to obtain the solution $S_{10}^9$ and then get 
\begin{equation}\label{Eqn54} 
\begin{array}{ll} 
V_{1010} = \!\!\!\! & \textstyle\frac{9}{8} \, G_4\, 
a_{211}^2 \big(a_{124}^2- \frac{5}{16} a_{211}^4 a_{124}+ \frac{429}{40960} 
 a_{211}^8 \big), \\[1.0ex] 
V_{1110} = \!\!\!\! & \textstyle\frac{9}{8} \, G_5\,
a_{211}^2 \big(a_{124}^2- \frac{5}{16} a_{211}^4 a_{124}+ \frac{429}{40960}
 a_{211}^8 \big), \\[1.0ex]
V_{1210} = \!\!\!\! & \textstyle\frac{9}{8} \, G_6\,
a_{211}^2 \big(a_{124}^2- \frac{5}{16} a_{211}^4 a_{124}+ \frac{429}{40960}
 a_{211}^8 \big), \quad {\det}_{10}^9 = F_2, 
\end{array} 
\end{equation} 
which gives two solutions leading to a center at ${\rm E_0}$, one of them 
is $a_{211} \!=\! 0$, which yields the same solution as that 
obtained in Case (A) at order $10$. Thus, choosing the non-generic solution 
$a_{211} \!=\! 0$ at this order leads to the generic pattern of Case (A) 
starting from  $\varepsilon^{11}$ order (i.e., for $k \ge 11$). 
The second solution, given by 
\begin{equation}\label{Eqn55} 
a_{124} =\textstyle\frac{1}{64} \left(10 \pm \frac{1}{10} \sqrt{5710} 
\right) a_{211}^4, 
\end{equation}  
is a generic solution for Case (B), different from the 
generic pattern of Case (A).  
Then, following a similar computation procedure as that used in 
Case (A), we obtain the generic solutions up to $\varepsilon^{39}$ order. 
The results are given in Table~\ref{table2}, showing a 
$9^6$-$10^6$-$11^1$ generic pattern, starting from order $14$. 
The notations used in this table are the same as that used in 
Table~\ref{table1}.
For each $k$, the key coefficient used to obtain the center condition is 
give below. 
$$ 
\begin{array}{lll} 
{\rm 9 \ LC:} & k=7 & a_{037}= - \frac{35}{768} a_{211} \big[
64 a_{123}^2
-a_{211} (15 a_{123} a_{211}^2 \\[0.5ex] 
& & \qquad \qquad \qquad \quad 
-16 a_{124} a_{211} -48 a_{123} a_{212}) \big] \\[0.5ex]   
& k=9  & a_{039}= \cdots  
\\ [0.5ex]  
& k=14 & a_{0314}= - \frac{35}{48} a_{212}^3 a_{128} \\  [0.5ex] 
& k=15 & a_{0315}= - \frac{35}{48} a_{212}^2 (a_{129} a_{212} 
+3 a_{128} a_{213})  \\  [0.5ex] 
& k=16 & a_{0316}=  - \frac{35}{768} a_{212} 
\big[48 a_{213}^2 a_{128}  
+a_{212} ( 16 a_{1210} a_{212} \\[0.5ex] 
& & \qquad \qquad  
+48 a_{129} a_{213}
+48 a_{128} a_{214}
- 15 a_{128} a_{212}^2 
) \big] \\  [0.5ex] 
& k=17$--$19 & a_{03k}= \cdots \\  [0.5ex] 
& k=27$--$32 & a_{03k}= \cdots \\  [1.5ex] 
{\rm 10 \ LC:} & k = 8 & a_{123} = 0 \\[0.5ex] 
& k=10  &  a_{124} = \frac{100 \pm \sqrt{5710}}{640}\, a_{211}^4  \\ [0.5ex] 
& k=11  &  a_{125} = \frac{100 \pm \sqrt{5710}}{10240} \, 
a_{211}^3 \big[ 64 a_{212} + a ( a^2-4) a_{211}^2 \big] \\ [0.5ex] 
& k=12  &  a_{126} = \cdots 
\\ [0.5ex] 
& k=20  &  a_{128} = \frac{100 \pm \sqrt{5710}}{640} \, a_{212}^4 \\ [0.5ex] 
& k=21  &  a_{129} = \frac{100 \pm \sqrt{5710}}{160} \, a_{212}^3 a_{213} 
\\ [0.5ex] 
& k=22$--$25  &  a_{12(k-12)} = \cdots \\ [0.5ex] 
\end{array}
$$
$$  
\begin{array}{lll} 
& k=33  &  a_{1215} = - \frac{100 \pm \sqrt{5710}}{10240} a_{213}  
 \big[a_{213}^2 (15 a_{213}^2 -64 a_{216}) \\[0.5ex] 
& & \qquad \qquad \qquad \qquad \qquad  
- 64 a_{214} ( a_{214}^2 +3 a_{213} a_{215}) \big]
\\ [0.5ex] 
& k=34$--$38 &  a_{12(k-18)} = \cdots \\[1.5ex] 
{\rm 11 \ LC:} & k=13m &  \\  
&  m= 1,2,3 & a_{21m} = 0   
\end{array}
$$

\begin{table}[!h]
\caption{Bifurcation of limit cycles for generic Case (B)}  
\centering 
\vspace*{0.15in} 
\label{table2}
\begin{tabular}{|c|c|c|c|}
\hline
$k$ & ($S_k^m,W_k^m$) & LC & Condition for Center   \\
\hline
 7 & $(S_7^{8},W_7^{8})$  & 9 & ${\rm S_{7c}^8} $ \\ 
 8 & $(S_8^9,W_8^9)$  & {\color{green}{10}} 
& ${\color{green}{\rm S_{8c}^9}} $  \\  
 9 & $(S_9^{8},W_9^{8})$  & 9 & ${\rm S_{9c}^8} $ \\ 
 10--12 & $(S_{10}^9,W_{10}^9)$  & {\color{green}{10}} 
& $ {\color{green}{\rm S_{kc}^9}} $\\  
            &  &  &  \vspace*{-0.17in}  \\ 
 13 & $(S_{13}^{10},W_{13}^{10})$ & {\color{red}{11}} 
& ${\color{red}{\rm S_{13c}^{10}}} $ \\ 
            &  &  &  \vspace*{-0.15in}  \\ \hline  
            &  &  &  \vspace*{-0.15in}  \\ 
 14--19 & $(S_{k}^{8},W_{k}^{8})$  & 9 & ${\rm S_{kc}^8} $ \\ 
 20--25 & $(S_{20}^9,W_{20}^9)$  & {\color{green}{10}} 
& ${\color{green}{\rm S_{kc}^9}} $ \\
            &  &  &  \vspace*{-0.17in}  \\ 
 26 & $(S_{26}^{10},W_{26}^{10})$ & {\color{red}{11}} 
& ${\color{red}{\rm S_{26c}^{10}}} $ \\ 
            &  &  &  \vspace*{-0.15in}  \\ \hline  
            &  &  &  \vspace*{-0.15in}  \\ 
 27--32 & $(S_{k}^{8},W_{k}^{8})$  & 9 & ${\rm S_{kc}^8} $ \\ 
 33--38 & $(S_{k}^{9},W_{k}^{9})$  & {\color{green}{10}} 
& $ {\color{green}{\rm S_{kc}^9}} $ \\ 
 39 & $(S_{26}^{10},W_{26}^{10})$ & {\color{red}{11}} 
&  ${\color{red}{\rm S_{39c}^{10}}} $ 
\vspace*{-0.15in} \\ 
& & & \\ \hline
\end{tabular}
\end{table}

\subsection{Non-generic solutions} 

Couple of non-generic solutions have been discussed in Case (B) 
(see Section 3.5), showing that setting $a_{211} \!=\! 0$ at order 
$8$ or $10$ (see Eqns.~\eqref{Eqn53} and \eqref{Eqn54})
leads to the $9^4$-$10^2$-$11^1$ generic pattern of Case (A) 
for orders greater than $10$ or $11$.
These two examples give a route from Case (B) to Case (A).  
In this section, we present several more non-generic solutions to 
show other possibilities that they eventually return to either 
the $9^4$-$10^2$-$11^1$ generic pattern of Case (A) or 
$9^6$-$10^6$-$11^1$ generic pattern of Case (B). 
Other cases can be similarly discussed. 
Since the discussions for different cases are similar, 
we will not give the details but list the cases below and 
summarize the results in Table~\ref{table3}. 

\begin{enumerate} 
\item[{(A1)}]  
In Case (A), at order $13$: $a_{212}=0$, leading to Case (A). 

\item[{(A2)}]  
In Case (A), at order $12$: $a_{124}=0$, leading to Case (B). 

\item[{(B1)}]  
In Case (B), at order $11$: $a_{211}=0$, leading to Case (B).  

\end{enumerate}

For each $k$, the key coefficient used to obtain the center condition is 
given below. 
$$ 
\!\! \begin{array}{lll} 
{\rm Case \ (A1) } & k=13 & a_{212}= 0 \\  
& k=14  & a_{0314}= -\frac{35 a_{125}}{48} \big[ 4 a_{125} a_{214} 
\!+\! a_{213} (8 a_{126} \!+\! a_{213}^2) \big]  \hspace*{0.8in} \\ [0.5ex]  
& k=16 & a_{0316}= - \frac{35}{48} \big[ a_{127} a_{213} 
(8 a_{126}+a_{213}^2) \\[0.5ex] 
& & \qquad \qquad \quad +a_{126} a_{214} (4 a_{126} +3 a_{213}^2)\big]  
\\  [0.5ex] 
& k=17 & a_{0317}= -\frac{35}{48} a_{128} a_{213} (a_{213}^2+8 a_{126}) 
\\[0.5ex] 
& & \qquad -\frac{35}{48} a_{127} (4 a_{213} a_{127} 
\!+\! 8 a_{214} a_{126} \!+\! 3 a_{214}a_{213}^2)  \\[0.5ex] 
&  & \qquad - \frac{35}{48} a_{126} (3 a_{214}^2 a_{213} 
\!+\! 4 a_{126} a_{215} \!+\! 3 a_{215} a_{213}^2) \\  [0.5ex] 
& k=15 & a_{125}= 0 \\  [0.5ex] 
& k=18 & a_{126}= -\frac{9}{8}\, a_{213}^2  
\end{array} 
$$
$$ 
\!\! \begin{array}{lll} 
{\rm Case \ (B1)} & k=11 & a_{211}= 0 \\[0.5ex] 
& k=12 & a_{0312} \!=\! -\frac{35}{48} a_{212} 
(a_{126} a_{212}^2 \!+\! 3 a_{125} a_{212} a_{213} \!+\! 4a_{125}^2) 
 \\[0.5ex] 
& k=13,15,17 & a_{03k} = \cdots  \\[0.5ex] 
& k=14,16,18 & a_{12(k/2-2)} =  0   \\ [1.5ex] 
{\rm Case \ (A2)} & k=13 & 
a_{0313} \!=\! -\frac{35}{48} a_{212}\big[ a_{127} a_{212}^2 
\!+\!  a_{126} (3 a_{213} a_{212} \!+\! 8 a_{125}) \big] \\[0.5ex] 
& & \qquad \quad +\frac{35}{768} a_{125} 
( 15 a_{212}^4 -48 a_{212}^2 a_{214} 
\\[0.5ex] 
& & \qquad \qquad \qquad \quad 
-48 a_{212} a_{213}^2 -64 a_{125} a_{213} ) \\[0.5ex] 
& k=15,17 & a_{03k} = \cdots \\[0.5ex] 
& k=12,14,16,18 & a_{12(k/2-2)} = 0  
\end{array}
$$

\begin{table}[!h]
\caption{Non-generic solutions}  
\centering 
\vspace*{0.15in} 
\label{table3}
\begin{tabular}{|c|c|c|c|c|}
\hline
\!\!\! Case \!\!\! &$k$ & ($S_k^m,W_k^m$) & LC & Condition for Center   \\
\hline
& 13,15 & $(S_{13}^9,W_{13}^9)$  & {\color{green}{10}} 
& ${\color{green}{\rm S_{kc}^9}} $  \\  
{\rm (A1)} & 14,16,17 & $(S_{14}^{8},W_{14}^{8})$  & 9 
& ${\rm S_{kc}^8} $ \\ 
& 18 & $(S_k^9,W_k^9)$  & {\color{green}{10}} 
& ${\color{green}{\rm S_{18c}^9}} \ \Longrightarrow \ $ 
{\color{red}{generic Case (A)}}  \\  
&            &  &  &  \vspace*{-0.15in}  \\ \hline  
&           &  &  &  \vspace*{-0.15in}  \\ 
& 11,14,16 & $(S_{11}^{9},W_{11}^{9})$  & {\color{green}{10}} 
& ${\color{green}{\rm S_{kc}^9}} $ \\ 
{\rm (B1)} & \!\!\! 12,13,15,17 \!\!\! & $(S_{k}^8,W_{k}^8)$  & 9 & ${\rm S_{kc}^8} $ \\
& 18 & $(S_{18}^{9},W_{18}^{9})$ & {\color{green}{10}} 
& ${\color{green}{{\rm S_{18c}^9}}} \ \Longrightarrow \ $ 
{\color{red}{generic Case (B)}}  \\ 
&            &  &  &  \vspace*{-0.15in}  \\ \hline  
&            &  &  &  \vspace*{-0.15in}  \\ 
&\!\!\! 12,14,16 \!\!\! & $(S_{k}^{9},W_{k}^{9})$  & {\color{green}{10}} 
& $ {\color{green}{\rm S_{kc}^9}} $ \\ 
{\rm (A2)} & \!\!\! 13,15,17 \!\!\! & $(S_{k}^{8},W_{k}^{8})$  & 9 & ${\rm S_{kc}^8}  $ \\ 
&  18 & $(S_{18}^{9},W_{18}^{9})$ & {\color{green}{10}} 
& ${\color{green}{\rm S_{18c}^9}} \ \Longrightarrow \ $ 
{\color{red}{generic Case (B)}}  \vspace*{-0.15in} \\ 
& & & & \\ \hline
\end{tabular}
\end{table}

Therefore, there are four possible routes for the non-generic solutions: 
from Case (A) to Case (A) or Case (B); and 
from Case (B) to Case (A) or Case (B).

\subsection{Summary of this section} 

Summarizing the results obtained in sections 
3.4, 3.5 and 3.6, we have the following theorem. 

\begin{theorem}\label{Thm8}  
For system \eqref{Eqn11}, based on the higher-order focus values, 
there exist two generic patterns: One is $9^4$-$10^2$-$11^1$ pattern 
starting from order $8$ with four consecutive $9$ limit cycles, 
followed by two consecutive $10$ limit cycles, and then 
one $11$ limit cycles up to $\varepsilon^{35}$ order; 
and the other is $9^6$-$10^6$-$11^1$ pattern, 
starting from order $14$ with six consecutive $9$ limit cycles, 
followed by six consecutive $10$ limit cycles, and then 
one $11$ limit cycles up to $\varepsilon^{39}$ order. 
Other non-generic solutions deviate from the current pattern for 
certain orders and eventually return to either 
the $9^4$-$10^2$-$11^1$ pattern or the $9^6$-$10^6$-$11^1$ pattern.   
\end{theorem}

Finally, we propose a conjecture on the number of limit cycles 
around ${\rm E_0}$ for system \eqref{Eqn11}. 

\vspace*{0.10in} 
\noindent 
{\bf Conjecture.} For the perturbed system \eqref{Eqn11}, 
the maximal number of small limit cycles which can bifurcate from the
equilibrium ${\rm E_0}$ is $11$,

\section{Conclusion}

In this paper, we have applied high-order focus value computation to 
prove that system \eqref{Eqn11}
can have $11$ limit cycles around the equilibrium of 
\eqref{Eqn11}, obtained by 
perturbing at least $\varepsilon^7$-order focus values. 
Moreover, no more than $11$ limit cycles can be found up to 
$\varepsilon^{39}$-order analysis. It is believed that 
system \eqref{Eqn11} can have maximal $11$ small limit cycles 
around the equilibrium.

\section*{Acknowledgment}
This work was supported by the National Natural Science Foundation of 
China (NSFC No.~11501370), and the  
Engineering Research Council of Canada (NSERC No.~R2686A02).





\section*{Appendix A} 

The coefficients $C_i$'s
in \eqref{Eqn31}, \eqref{Eqn35} and \eqref{Eqn38} are given 
below. 
$$
\!\!\! \begin{array}{rll} 
C_1 = &\!\!\! - \frac{3}{16} (3 a^4 + 4 a^2 +44) & 
C_2 = -\, \frac{a}{48} (3 a^4+12 a^2 + 116) \\[1.0ex] 
C_3 = &\!\!\! -\, \frac{1}{8} (a^4 +2 a^2 +5) & 
C_4 = -\, \frac{1}{16} (9 a^4 - 20 a^2 + 172) \\[1.0ex] 
C_5 = &\!\!\! -\, \frac{3\,a}{8} (3 a^2 -8 )  & 
C_6 = -\,\frac{a}{12} ( a^2 -120 )  \\[1.0ex] 
C_7 = &\!\!\! - \frac{1}{8} (3 a^2 -80)  & 
C_8 = \frac{3}{128} a (7a^4+68a^2-900) \\[1.0ex]
C_9 = & \!\!\! \frac{1}{64}(3 a^6 \!+\! 40 a^4 \!-\! 860 a^2 \!-\! 1600) &  
C_{10} = -\, \textstyle\frac{a}{576} ( 303 a^4 + 1596 a^2 + 8096) \\[1.0ex] 
C_{11} = & \!\!\! - \,\frac{a}{64} (55 a^2 - 256) & 
C_{12} = -\, \frac{1}{384} (569 a^4-1660 a^2+1420) \\[1.0ex] 
C_{13} = &\!\!\! \frac{1}{192} (21 a^6 \!+\! 80 a^4 \!-\! 2996 a^2 \!-\!9040) & 
C_{14} = \frac{3}{64} (7 a^4 + 8 a^2 + 160) \\[1.0ex] 
C_{15} = &\!\!\! \frac{a}{128} (49 a^4 + 220 a^2 - 380) & 
C_{16} = \frac{a}{32} (9a^4+36a^2-172) \\[1.0ex] 
C_{17} = &\!\!\! \frac{3}{64} (15 a^4+28 a^2-916) & 
C_{18} = -\, \textstyle\frac{1}{16} (3 a^4 + 4 a^2 + 340) 
\end{array}
$$
$$ 
\!\!\! \begin{array}{rll} 
C_{19} = &\!\!\! -\, \frac{3a}{8} (a^2 - 16) & 
C_{20} = \frac{a}{96}(45 a^4 + 266 a^2 - 644) \\[1.0ex] 
C_{21} = &\!\!\! \frac{a}{32} (41 a^2+72) & 
C_{22} = \frac{1}{192} (303 a^4+1728 a^2-3620) \\[1.0ex] 
C_{23} = &\!\!\! -\frac{9}{512} a ( a^4 + 12 a^2 + 200) & 
C_{24} = -\, \frac{9}{256} (a^4+160) \\[1.0ex] 
C_{25} = & \!\!\! -\, \frac{1}{512} (70 a^6 \!-\! 471 a^4 \!+\! 128 a^2 
\!-\! 300) & 
C_{26} = -\, \frac{9}{512} a (a^4 - 108 a^2 + 696) \\[1.0ex] 
C_{27} = & \!\!\! \frac{3}{256} (21 a^4 + 58 a^2 - 1880) & 
C_{28} = a_{125} a_{212}+a_{124} a_{213} \\[1.0ex] 
C_{29} = & \!\!\! a_{215} a_{212}+a_{214} a_{213} & 
C_{30} = 3 a_{213} a_{212}^2 
\end{array}
$$ 

\section*{Appendix B} 

The coefficients $s_i$'s involved in $H_2$ (see Eq. \eqref{Eqn36}) 
and $t_i$'s involved in $H_3$ (see Eq. \eqref{Eqn39}) are given 
as follows. 
\begin{equation*}
\begin{array}{rl}
s_1 =&\!\!\!\! \frac{1}{3072}a^6(3a^2+16)(10+4y+4x^2+3x^4)
-\frac{1}{192}(2850+1824a^2-85a^4)\\[1ex]
&\!\!\!\!
+\frac{a}{18}(6a^2 \!-\! 319)x
\!-\! \frac{1}{288}(5902 \!+\! 1568a^2 \!+\! 45a^4)y
-\frac{a}{9}(53+10x^2)xy
+\frac{13}{6}y^2
\\[1ex]
&\!\!\!\!
-\frac{1}{72}a^2(3a^2+4)y(y-x^2)
-\frac{1}{96}(1074+200a^2+45a^4)x^2
-\frac{29}{18}x^2y
-\frac{2}{9}y^3
\\[1ex]
&\!\!\!\!+\frac{a}{36}(12-4a^2+3a^4)x^3
-\frac{1}{1152}(6746+3140a^2+91a^4+3a^6)x^4
+\frac{2}{3}x^2y^2 \\[1ex]
s_2 =&\!\!\!\!
\frac{1}{128}a^2(24-10a^2+5a^4)
-\frac{1}{64}(1120-24a^2+ 2a^4 -a^6)(y+x^2)
-\frac{1}{2}x^2y
\\[1ex]
&\!\!\!\!
+\frac{a}{16}(4 \!+\! 5a^2)x
\!-\! \frac{25}{4}x^2
\!-\! \frac{a}{256}(a^2 \!-\! 4)x [32(y \!-\! x^2) \!-\! 3a^3x^3]
\!-\! \frac{1}{8}(73 \!-\! 2a^2 )x^4 \\[1ex]
s_3 =&\!\!\!\! \frac{1}{384}a(-4716+376a^2+25a^4+15a^6)
-\frac{1}{96}(2260-164a^2-15a^4)x\\[1ex]
&\!\!\!\!+\frac{1}{192}a^5(5+3a^2)(y+x^2)
-\frac{1}{12}a(a^2-4)y(y-x^2)
-\frac{2}{3}xy^2
+2x^3y \\[1ex]
&\!\!\!\!
-\frac{1}{48}[a(1151+10a^2)y
-a(1511+40a^2)x^2
+(100-4a^2-3a^4)xy \\[1ex]
&\!\!\!\! +(140 -76a^2 +11a^4)x^3]
-\frac{1}{768}a(10216+100a^2-2a^4-9a^6)x^4  \\[1.0ex] 
s_4 =&\!\!\!\! \frac{1}{9216}a(96656+17952a^2+3640a^4+24a^6+9a^8)\\[1ex]
&\!\!\!\!-\frac{5}{4608}(2256+7272a^2-56a^4+6a^6+9a^8)x
-\frac{5}{72}a(224+8a^2+3a^4)y \\[1ex]
&\!\!\!\! -\frac{5}{144}a(642-8a^2-3a^4)x^2
-\frac{5}{1152}(368-1628a^2-92a^4+3a^6)xy \\[1ex]
&\!\!\!\! -\frac{5}{2}ay^2
+\frac{5}{1152}(5272+86a^4+2964a^2-3a^6)x^3
-\frac{275}{36}ax^2y 
\\[1ex]
&\!\!\!\! 
+\frac{5}{288}(372 +4a^2 +3a^4) xy^2
-\frac{5}{144}a(140+12a^2+3a^4)x^4
\\[1ex]
&\!\!\!\! 
-\frac{5}{288}(268+3a^4+4a^2)x^3y
-\frac{5}{18}xy^3-\frac{25}{18}ax^4y
+\frac{5}{6}x^3y^2 \\[1ex]
s_5 =&\!\!\!\! 
\frac{1}{256}a(-16a^2-8a^4+a^6+1120)
+\frac{5}{256}(560+9a^4-2a^6-4a^2)x\\[1ex]
&\!\!\!\!-\frac{5}{64}a(20-11a^2)x^2
-\frac{5}{128}a^2(3a^2+4)x(y+x^2)
+\frac{175}{8}xy
+\frac{265}{16}x^3\\[1ex]
&\!\!\!\!+\frac{5}{32}a(a^2-4)x^2(y-x^2)
-\frac{5}{8}x^3y \\[1ex]
\end{array}
\end{equation*}
\begin{equation*}
\begin{array}{rl}
s_6 =&\!\!\!\! 
\frac{1}{768}a^2(540a^2 \!+\! 3904 \!-\! 8a^4 \!+\! 3a^6)
-\frac{5}{1536}a(484 \!-\! 100a^2 \!-\! 23a^4 \!+\! 12a^6)x\\[1ex]
&\!\!\!\!-\frac{5}{12}a^2(a^2 \!-\! 1)y
\!-\! \frac{35}{384}(244 \!-\! 4a^2 \!-\! 7a^4)x^2
\!+\! \frac{5}{768}a(5132 \!+\! 84a^2 \!-\! 13a^4)xy\\[1ex]
&\!\!\!\!-\frac{5}{768}a(-3756-76a^2+13a^4)x^3
+\frac{5}{192}(380+3a^4+4a^2)x^2y
\\[1ex]
&\!\!\!\! 
+\frac{5}{48}a(a^3-4)xy(y-x^2)
-\frac{5}{192}(11a^4+20a^2+12)x^4
-\frac{5}{6}x^2y^2+\frac{5}{2}x^4y.
\end{array}
\end{equation*}
\begin{equation*}
\begin{array}{rl}
t_1 =&\!\!\!\!
\frac{1}{384}a(-4716+376a^2+25a^4+15a^6)
-\frac{1}{96}(2260-164a^2-15a^4)x
\\[1ex]
&\!\!\!\!
+\frac{1}{192}a^5(5 +3a^2)(y+x^2)
-\frac{1}{48}a( 1151 +10a^2 )y
-\frac{2}{3}xy^2
+2x^3y
\\[1ex]
&\!\!\!\!
-\frac{1}{48}a( 1511 \!+\! 40a^2 )x^2
+\frac{1}{48}(100 \!-\! 4a^2 \!-\! 3a^4)xy
-\frac{1}{12}a(a^2-4)y(y-x^2)
\\[1ex]
&\!\!\!\!
+\frac{1}{48}( 140 -76a^2 +11a^4)x^3
+\frac{1}{768}a(-100a^2-10216+9a^6+2a^4)x^4 \\[1ex]
t_2 =&\!\!\!\! 
\frac{1}{64}a^2( 24 -10a^2 +5a^4)
-\frac{1}{32}(1120 -24a^2 +2a^4 -a^6)(y+x^2)
\\[1ex]
&\!\!\!\!
+\frac{1}{8}a(5a^2+4)x
-\frac{25}{2}x^2
-\frac{1}{4}a(a^2-4)x(y-x^2)
-x^2y
\\[1ex]
&\!\!\!\!
-\frac{1}{128}(2336 -64a^2 +12a^4 -3a^6)x^4 \\[1ex]
t_3 =&\!\!\!\!
-\frac{1}{2048}a^5(20+a^4)(5+2y+2x^2+2x^4)
-\frac{1}{128}a(550-343a^2)
\\[1ex]
&\!\!\!\!
-\frac{1}{256}(6800 \!-\! 24a^2 \!-\! 10a^4 \!+\! 5a^6)x
-\frac{3}{64}a(25a^2 \!+\! 34)y
-\frac{1}{32}a(45a^2 \!+\! 61)x^2
\\[1ex]
&\!\!\!\!
-\frac{1}{1024}a^4(2-a^2)x(8y-8x^2+3ax^3)
-\frac{3}{16}(50-a^2)xy
+\frac{3}{16}(80-a^2)x^3
\\[1ex]
&\!\!\!\!
+\frac{1}{32}a(a^2-4)x^2y
-\frac{1}{64}a( 197 -2a^2 )x^4
+\frac{1}{4}x^3y \\[1ex]
t_4 =&\!\!\!\!
\frac{1}{768}a^2(3904 \!+\! 540a^2 \!-\! 8a^4 \!+\! 3a^6)
-\frac{5}{1536}a(484 \!-\! 100a^2 \!-\! 23a^4 \!+\! 12a^6)x
\\[1ex]
&\!\!\!\!
-\frac{5}{12}a^2(a^2-1)y
-\frac{35}{384}(244-7a^4-4a^2)x^2
\\[1ex]
&\!\!\!\!
-\frac{5}{768}a(-5132-84a^2+13a^4)xy
-\frac{5}{768}a(-3756-76a^2+13a^4)x^3
\\[1ex]
\end{array}
\end{equation*}
\begin{equation*}
\begin{array}{rl}
&\!\!\!\!
+\frac{5}{192}(380+4a^2+3a^4)x^2y
+\frac{5}{48}a(a^2-4)xy(y-x^2)
\\[1ex]
&\!\!\!\!
-\frac{5}{192}(12+20a^2+11a^4)x^4
-\frac{5}{6}x^2y^2
+\frac{5}{2}x^4y \\[1ex]
t_5 =&\!\!\!\!
\frac{1}{128}a(1120-16a^2-8a^4+a^6)
+\frac{5}{128}(560-4a^2+9a^4-2a^6)x
\\[1ex]
&\!\!\!\!
+\frac{5}{32}a(11a^2-20)x^2
-\frac{5}{64}a^2(4+3a^2)x(y+x^2)
+\frac{175}{4}xy
\\[1ex]
&\!\!\!\!
+\frac{265}{8}x^3
+\frac{5}{16}a(a^2-4)x^2(y-x^2)
-\frac{5}{4}x^3y \\[1ex]
t_6 =&\!\!\!\!
-\frac{1}{2048}a^2(2544 \!-\! 2336a^2 \!+\! 8a^4 \!+\! 3a^6)
-\frac{5}{2048}a(2632 \!+\! 276a^2 \!+\! 3a^4 \!-\! 8a^6)x
\\[1ex]
&\!\!\!\!
+\frac{5}{512}(1240-4a^2-41a^4+4a^6)x^2
+\frac{5}{1024}a(-1832+892a^2+a^4)xy
\\[1ex]
&\!\!\!\!
+\frac{5}{1024}a(-1368+696a^2+a^4)x^3
-\frac{5}{256}a^2(3a^2+4)x^2(y-x^2)
\\[1ex]
&\!\!\!\!
+\frac{375}{32}x^2y
-\frac{65}{16}x^4
-\frac{5}{128}a(a^2-4)x^3y
+\frac{5}{16}x^4y
\end{array}
\end{equation*}

\end{document}